\documentclass[11pt,english]{article}

\usepackage{amsfonts,amsmath,amsthm,amscd,amssymb,latexsym,cite,verbatim}

\theoremstyle{plain}
\newtheorem{theorem}{Theorem}[section]
\newtheorem{lemma}{Lemma}[section]

\newtheorem{corollary}{Corollary}[section]

\newtheorem{definition}{Definition}[section]
\theoremstyle{definition}

\newtheorem{remark}{Remark}[section]
\newcommand{\keywords}{\textbf{Key words. }\medskip}
\newcommand{\subjclass}{\textbf{MSC 2010. }\medskip}
\renewcommand{\abstract}{\textbf{Abstract. }\medskip}

\numberwithin{equation}{section}
\begin{document}

\title{Localized peaking regimes for quasilinear parabolic equations}

\author{Andrey~E.~Shishkov and Yevgeniia~A.~Yevgenieva}

\date{}

\maketitle

\begin{abstract}
 This paper deals with the asymptotic behavior as $t\rightarrow T<\infty$ of all weak
(energy) solutions of a class of equations with the following model representative:
\begin{equation*}
    (|u|^{p-1}u)_t-\Delta_p(u)+b(t,x)|u|^{\lambda-1}u=0 \quad
(t,x)\in(0,T)\times\Omega,\,\Omega\in{R}^n,\,n>1,
\end{equation*}
with prescribed global energy function
\begin{equation*}
E(t):=\int_{\Omega}|u(t,x)|^{p+1}dx+
\int_0^t\int_{\Omega}|\nabla_xu(\tau,x)|^{p+1}dxd\tau \rightarrow\infty\ \text{ as
}t\rightarrow T.
\end{equation*}
Here $\Delta_p(u)=\sum_{i=1}^n\left(|\nabla_xu|^{p-1}u_{x_i}\right)_{x_i}$, $p>0$,
$\lambda>p$, $\Omega$ is a bounded smooth domain, $b(t,x)\geq0$. Particularly, in the
case
\begin{equation*}
E(t)\leq
F_\mu(t)=\exp\left(\omega(T-t)^{-\frac1{p+\mu}}\right)\quad\forall\,t<T,\,\mu>0,\,\omega>0,
\end{equation*}
it is proved that solution $u$ remains uniformly bounded as $t\rightarrow T$ in an
arbitrary subdomain $\Omega_0\subset\Omega:\overline{\Omega}_0\subset\Omega$ and the
sharp upper estimate of $u(t,x)$ when $t\rightarrow T$ has been obtained depending on
$\mu>0$ and $s=dist(x,\partial\Omega)$. In the case $b(t,x)>0$
$\forall\,(t,x)\in(0,T)\times\Omega$ sharp sufficient conditions on degeneration of
$b(t,x)$ near $t=T$ that guarantee mentioned above boundedness for arbitrary (even
large) solution have been found and the sharp upper estimate of a final profile of
solution when $t\rightarrow T$ has been obtained.
\end{abstract}



\section{Introduction and formulation of main results}

Let $\Omega$ be a bounded domain in ${R}^n$, $n\geq1$
 with $C^2$--smooth boundary $\partial\Omega$.
 We will consider the set of all weak solutions $u(t,x)$ of the following initial
 value problem in the cylindrical domain $Q=(0,T)\times\Omega$,
 $0<T<\infty$:
\begin{equation}\label{equat}
    P_p(u):=(|u|^{p-1}u)_t-\sum_{i=1}^n(a_i(t,x,u,\nabla u))_{x_i}=0 \ \ \ \text{ in }
    Q,\ \ p=\text{const}>0,
\end{equation}
\begin{equation}\label{initial}
u(0,x)=u_0\ \ \text{in} \ \Omega,\ \ u_0\in L^{p+1}(\Omega),
\end{equation}
Here
$a_i(t,x,s,\xi)$, $i=1,2,...,n$, are continuous functions satisfying the following
coercitivity and growth conditions:
\begin{equation}\label{coercit}
d_0|\xi|^{p+1}\leq\sum_{i=1}^na_i(t,x,s,\xi)\xi_i;\ \ \
\forall(t,x,s,\xi)\in\bar{Q}\times{R}^1\times{R}^n;\ d_0=\textrm{const}>0;
\end{equation}
\begin{equation}\label{growth}
|a_i(t,x,s,\xi)|\leq d_1|\xi|^p;\ \ \
\forall(t,x,s,\xi)\in\bar{Q}\times{R}^1\times{R}^n;\ i=1, ..., n;\
d_1=\textrm{const}<\infty.
\end{equation}

\begin{definition}\label{Def1}
We will call function $u(t,x)\in C_{loc}([0,T);L^{p+1}(\Omega))$ a weak (energy)
solution of the problem \eqref{equat}--\eqref{initial} if
\begin{enumerate}
\item[i)] $u(t,\cdot)\in L^{p+1}_{loc}([0,T);W^{1,p+1}(\Omega))$;

\item[ii)] $(|u(t,\cdot)|^{p-1}u(t,\cdot))_t\in L^{\frac{p+1}p}_{loc}
([0,T);(W^{1,p+1}_0(\Omega))^*)$;

\item[iii)] the following integral identity
\begin{equation}\label{ident1}
\int_0^\tau\langle(|u|^{p-1}u)_t,\eta\rangle dx+
\int_0^\tau\int_{\Omega_1}\sum_{i=1}^na_i(t,x,u,\nabla u)\eta_{x_i}dxdt=0
\end{equation}
holds for an arbitrary function $\eta(t,\cdot)\in L^{p+1}((0,\tau);W^{1,p+1}_0(\Omega))$
\ with an arbitrary $\tau<~T$;

\item[iv)] initial condition \eqref{initial} is satisfied.
\end{enumerate}
Here as generally accepted $W^{1,p+1}_0(\Omega)$ \ is a closure in the norm of
$W^{1,p+1}(\Omega)$ \ of a set of smooth functions $f$, such that $f=0$ on
$\Gamma\in\partial\Omega$.
\end{definition}

We consider the set of all weak solutions $u$ of the problem
\eqref{equat}--\eqref{initial} which have $t=T$ as blow-up time in the sense that
\begin{equation}\label{energy}
\begin{aligned}
&E(t)+h(t):=E^{(u)}(t)+h^{(u)}(t):=\\&:=\int_0^t\int_{\Omega}|\nabla_xu(\tau,x)|^{p+1}dxd\tau+
\sup_{0<\tau<t}\int_{\Omega}|u(\tau,x)|^{p+1}dx\rightarrow\infty\ \text{ as
}t\rightarrow T.
\end{aligned}
\end{equation}
Namely for all class of solutions $u$ satisfying the following estimate
\begin{equation}\label{class1}
E^{(u)}(t)+h^{(u)}(t)\leq F(t)\ \ \ \forall t<T
\end{equation}
with an arbitrary prescribed nondecreasing function $F(t):F(t)\rightarrow\infty$ as
$t\rightarrow T$, we have obtained the precise upper energy estimate of a solution $u$
near the blow-up time $T$ depending on $F(t)$. It is worth to mention that the described
unbounded growth of a solution may be generated in various ways, i.e. by boundary regime
with the infinite peaking:
\begin{equation}\label{blow1}
    u(t,x)\Big|_{\partial\Omega}=f(t,x)\rightarrow\infty\ \ \text{ as } t\rightarrow
    T,
\end{equation}
or
\begin{equation}\label{blowN}
\frac{\partial u(t,x)}{\partial
N}\Biggr|_{\partial\Omega}:=\sum_{i=1}^na_i(t,x,u,\nabla_x
u)\nu_i=g(t,x)\rightarrow\infty\ \text{ as }t\rightarrow T.
\end{equation}
where $\nu=\nu(x):=(\nu_1, ..., \nu_n)$ is a unit vector of an outward normal to
$\partial\Omega$ at a point $x$. Using different self-similar solutions or integral
representation of a solution and the barrier techniques, asymptotic and localization
properties of solutions of various linear and quasilinear second order parabolic
equations with boundary peaking regimes \eqref{blow1}, \eqref{blowN} was studied by many
authors (see \cite{SGKM1}, \cite{GH1}, \cite{CE1}, \cite{GG1}, \cite{Ven1} and
references therein).

Another situation with the unbounded growth of solution's energy near $t=T$ occurs when
a solution $u$ of the problem \eqref{equat}--\eqref{initial} is the restriction on
domain $Q$ of a solution $v(t,x)$ of the equation
\begin{equation}\label{equat2}
P_p(v)=\varphi(t,x)\ \ \text{ in
}Q_1=[0,T)\times\Omega_1,\,\overline{\Omega}\subset\Omega_1,
\end{equation}
($P_p(\cdot)$ is from \eqref{equat}) with a boundary condition like \eqref{blow1} or
\eqref{blowN} on $\partial\Omega_1$ with bounded boundary data $f$ or $g$ and initial
condition $v(0,x)=v_0(x)\ \ \ \forall x\in\Omega_1$, where $v_0(x)=u_0(x)$ $\forall
x\in\Omega$ and
\begin{equation}\label{force1}
\text{supp}\,\varphi(t,\cdot)\in\overline{\Omega_1\setminus\Omega}\ \ \ \forall t<T,\
\varphi(t,x)\rightarrow\infty\ \text{ as }t\rightarrow T.
\end{equation}

In \cite{ShSh1} (see also \cite{GSh1}, \cite{GSh2}, \cite{KSSh} and references therein)
  some variant of the local energy estimate method for
the study of the localization of peaking regimes was proposed and developed. This method
does not use any comparison techniques and is applicable for a very large class of
equations, including higher order quasilinear parabolic and pseudoparabolic equations.
Moreover this method comprises all classes of peaking regimes, including \eqref{blow1},
\eqref{blowN} and \eqref{class1}. So, it was proved (see Th.1.1 in \cite{GSh1} and
Th.6.4.1 in \cite{KSSh}) that for an arbitrary solution $u$ of problem \eqref{equat},
\eqref{initial} that satisfies condition \eqref{energy}, \eqref{class1} with
\begin{equation}\label{force2}
F(t)\leq F_0(t):=\exp\big(\omega(T-t)^{-\frac1{p}}\big)\quad
\forall\,t<T,\,\omega=const>0,
\end{equation}
the following property takes place: there exist constants $c>0$, $C<\infty$ depending on
$n,p$ only such that:
\begin{equation}\label{reg1}
\begin{aligned}
&h^{(u)}(t,s)+E^{(u)}(t,s)=h(t,s)+E(t,s):=\int_{\Omega(s)}|u(t,x)|^{p+1}dx+
\\&+\int_0^t\int_{\Omega(s)}|\nabla_xu(\tau,x)|^{p+1}dxd\tau\leq C<\infty\ \quad
\forall\,t<T,\,\forall\,s>c\omega^\frac p{p+1},
\end{aligned}
\end{equation}
where $\Omega(s)=\{x\in\Omega:d(x):=dist(x,\partial\Omega)>s\}$. Thus peaking regime
\eqref{force2} is localized blow-up regime with regional blow-up ("energy" analog of
"point-wise" definition (see \cite{SGKM1})) and with blow-up set
$\Omega_{bl}\subset\Omega\setminus\Omega\left(c\omega^\frac p{p+1}\right)$. For
nonlocalized HS-regimes (regimes \eqref{force2} with $\omega=\omega(t)\rightarrow\infty$
as $t\rightarrow T$) it was shown (see \cite{GSh2}, \cite{KSSh}) that
$\Omega_{bl}=\Omega$ in general case and upper estimates of the propagation of a
corresponding blow-up wave was obtained. For LS-regimes, i.e. regimes \eqref{force2}
with $\omega=\omega(t)\rightarrow0$ as $t\rightarrow T$ (for such regimes
$\Omega_{bl}\subset\partial\Omega$), there is no the precise description of the limiting
profile of a solution $u$ when $t\rightarrow T$ depending on function $\omega(t)$. In
this context we have obtained the following result.
\begin{theorem}\label{Th.1}
    Let $u(t,x)$ be an arbitrary energy solution of the problem
    \eqref{equat}--\eqref{initial} that satisfies the following global energy estimate:
    \begin{equation}\label{global}
        E(t)+h(t)\leq F_\mu(t):=\exp\left(\omega(T-t)^{-\frac1{p+\mu}}\right)\ \ \ \forall t<T,
    \end{equation}
where $E(t)$, $h(t)$ were defined in \eqref{energy}, $\omega>0$, $\mu>0$ are arbitrary
constants. Then there exist constants $c_1<\infty$, $c_2<\infty$, $c_3<\infty$ that
depend on $p,n,d_0,d_1$ only such that the following uniform with respect to $t<T$ a
priori estimate holds:
\begin{equation}\label{main1}
h(t,s)+E(t,s)\leq c_1\exp\big(c_2\omega^\frac{p+\mu}\mu s^{-\frac{p+1}\mu}\big)\ \
\forall\,s:0<s<s'_0:=c_3\min\left(1,\omega^\frac{p+\mu}{p+1}\right),
\end{equation}
where $h(t,s)$, $E(t,s)$ are the energy functions from \eqref{reg1}.
\end{theorem}

{\bf Conjecture.} Estimate \eqref{main1} is sharp with respect to parameters $\omega$,
$\mu$.

Next we will demonstrate the application of Theorem \ref{Th.1} for the study of
asymptotic properties of large solutions of a some class of quasilinear parabolic
equations of diffusion -- nonlinear degenerate absorption type. Namely let us consider
the problem:
\begin{equation}\label{da1}
P_p(u)=-b(t,x)|u|^{\lambda-1}u\ \ \text{ in }Q=(0,T)\times\Omega,\ T<\infty,\
\lambda>p>0,
\end{equation}
\begin{equation}\label{bound1}
u=\infty\ \ \ \ \text{ on }(0,T)\times\partial\Omega,
\end{equation}
\begin{equation}\label{initial1}
u=\infty\ \ \ \ \text{ on }\{0\}\times\Omega,
\end{equation}
where $\Omega$ is a domain from Theorem \ref{Th.1} and $b(t,x)$ (the absorption
potential) is a continuous function in $[0,T]\times\overline{\Omega}$ satisfying the
following condition:
\begin{equation}\label{abs1}
b(t,x)>0\ \ \text{ in }[0,T)\times\overline{\Omega},\ \ b(t,x)=0\ \ \text{ on
}\{T\}\times\Omega.
\end{equation}
If $p=1$ and $a_i(t,x,u,\xi)=\xi_i$, $i=1, ..., n$ then under the condition
\begin{equation}\label{abs2}
a_1(t)d(x)^\beta\leq b(t,x)\leq a_2(t)d(x)^\beta\ \ \ \forall (t,x)\in
[0,T]\times\Omega,\ \beta>-2,
\end{equation}
where $a_1(t)$, $a_2(t)$ is positive continuous on $[0,T)$ functions, the existence of
maximal $\overline{u}$ and minimal $\underline{u}$ positive solutions of the problem
\eqref{da1}, \eqref{bound1}, \eqref{initial1} was proved in \cite{DPP1}. Moreover the
main result of \cite{DPP1} says that under the following additional condition on the
degeneration of $a_1(t)$ near $t=T$:
\begin{equation}\label{abs3}
a_1(t)\geq c_0(T-t)^\theta\ \ \text{ in }[0,T),\ c_0=const>0,\ \theta=const>0
\end{equation}
for any $t_0\in(0,T)$ there exists $C=C(t_0)<\infty$ such that:
\begin{equation}\label{estim1}
\overline{u}(t,x)\leq
C\min\big\{(T-t)^{-\frac\theta{\lambda-1}},d(x)^{-\frac{2\theta}{\lambda-1}}\big\}d(x)^{-\frac{2+\beta}{\lambda-1}}\
\ \forall(t,x)\in[t_0,T)\times\Omega.
\end{equation}
In \cite{Sh1} the sharp sufficient flatness condition for $a_1(t)$ that guarantees the
boundedness of $\limsup_{t\rightarrow T}u(t,x)$ $\forall x\in\Omega$ for an arbitrary
solution $u$ of the problem under consideration was found. Namely it was proved that the
condition
\begin{equation}\label{abs4}
a_1(t)\geq c_0\exp\biggr(-\frac{\omega_0}{(T-t)}\biggr)\ \ \text{ in }[0,T),\
c_0=const>0,\ \omega_0=const>0
\end{equation}
guarantees the existence of a constant $k>0$, that does not depend on $\omega_0$ such
that
\begin{equation}\label{estim2}
\limsup_{t\rightarrow T}u(t,x)\leq C<\infty\ \ \ \forall
x\in{\Omega}_0:=\{x\in\Omega:d(x)>k\omega_0^\frac12\}.
\end{equation}
Here we prove the following statement.
\begin{theorem}\label{Th.2}
Let $u$ be an arbitrary weak (energy) solution (see definition \ref{def2}) of equation
\eqref{da1}, where the absorption potential $b(t,x)$ satisfies the condition
\begin{equation}\label{abs3}
a_1(t)g_1(d(x))\leq b(t,x)\leq a_2(t)g_2(d(x))\qquad \forall\,(t,x)\in
[0,T)\times\Omega.
\end{equation}
Here $g_1(s)\leq g_2(s)$ are arbitrary nondecreasing positive for all $s>0$ functions
and $a_1(t)$ satisfies the condition:
\begin{equation}\label{abs5}
a_1(t)\geq c_0\exp \left(-\omega(T-t)^{-\frac1{p+\mu}}\right)\quad
\forall\,t<T,\,c_0>0,\,\omega=const>0,\,\mu=const>0,
\end{equation}
Then the following estimate holds for all $\frac T2<t<T$:
\begin{equation}\label{bound08}
\begin{aligned}
&h(t,s)+E(t,s):=\int_{\Omega(s)}|u(t,x)|^{p+1}dx+\int_{\frac T2}^t\int_{\Omega(s)}
|\nabla_xu(\tau,x)|^{p+1}dxd\tau\leq
\\&\leq
K_1\min_{0<\bar{s}<s}\left\{\exp\left(K_2\omega^\frac{p+\mu}\mu(s-\bar{s})^{-\frac{p+1}\mu}\right)\cdot
\left(\int_0^{\bar{s}}g_1(h)^\frac{p+1}{1+p(\lambda+2)}dh\right)^{-\frac{1+p(\lambda+2)}{\lambda-p}}\right\}\quad\forall\,s\in(0,s'_0),
\end{aligned}
\end{equation}
where constants $K_1<\infty$, $K_2<\infty$ depend on known parameters of the problem
under consideration only, $s'_0$ is from \eqref{main1}.
\end{theorem}

\begin{corollary}
Let $g_1(s)=as^\nu$, $a>0$, $\nu>0$. Then estimate \eqref{bound08} yields (see example 2
in \S 5):
\begin{equation}\label{21.06}
\begin{aligned}
&h(t,s)+E(t,s)\leq K_3 \exp\left(K_2(1-\rho)^{-\frac{p+1}\mu}\omega^\frac{p+\mu}\mu
s^{-\frac{p+1}\mu}\right)\rho^{-b} s^{-b}\quad\forall\,t\in\left(\frac
T2,T\right),\,\forall\,s\in(0,s'_0),\\& \text{where }\
b=\frac{(\nu+1)(p+1)+p(\lambda+1)}{\lambda-p},\quad
K_3=K_3(a,\nu)=K_1a^{-\frac{p+1}{\lambda-p}}
\left(1+\frac{\nu(p+1)}{1+p(\lambda+2)}\right)^\frac{1+p(\lambda+2)}{\lambda-p}.
\end{aligned}
\end{equation}
Here $\rho=\rho\left(\frac{\mu
bs^\frac{p+1}\mu}{K_2\omega^{\frac{p+\mu}\mu}(p+1)}\right)$, where $\rho(\tau)$ is a
function defined by an optimizing condition. Namely this function satisfies the
following equation:
\begin{equation*}
\rho-\tau(1-\rho)^{\frac{p+\mu+1}\mu}=0,\quad\forall\tau>0.
\end{equation*}
It is easy to see that $\rho(\cdot):(0,\infty)\rightarrow(0,1)$ is a monotonically
increasing function. Moreover, $\rho(\tau)\tau^{-1}\rightarrow1$ as $\tau\rightarrow0$.
\end{corollary}

\begin{remark}\label{Rem2}
In the forthcoming paper we are going to consider the problem
\eqref{da1}--\eqref{initial1} under the condition that the absorption potential $b(t,x)$
degenerates on some manifold $\Gamma\subset\overline{\Omega}:b(t,x)\rightarrow0$ as
$t\rightarrow T,\,x\in\Gamma;\,b(t,x)>0\,\forall\,t\leq
T,\,\forall\,x\in\Omega\setminus\Gamma$. We are going to describe the propagation of
singularities of a large solution along $\Gamma$ and obtain sharp estimates of the
limiting profile of a solution near $t=T$ depending on the asymptotic of $b(t,x)$ near
$\{T\}\times\Gamma$.
\end{remark}

\section{Systems of differential inequalities with respect to the families of energy functions.}

Let us introduce the following families of subdomains of the domains $\Omega$, $Q$ from
\eqref{equat}, \eqref{initial}:
\begin{equation}\label{fam1}
    \Omega(s):=\{x\in\Omega:d(x)>s\},\ \ Q(s):=(0,T)\times\Omega(s)\ \ \ \forall\,s:0<s<s_\Omega,
\end{equation}
where $s_\Omega>0$ is such a constant that function $d(\cdot)\in
C^2(\Omega\setminus\Omega(s))$ $\forall s\leq s_\Omega$ and correspondingly
$\partial\Omega(s)$ is $C^2$--smooth manifold for all $0<s\leq s_\Omega$. As is well
known, the existence of such a constant follows from the prescribed smoothness of
$\partial\Omega$. Let us introduce  the following energy functions for arbitrary
$a,b:0\leq a<b<T$ connected with solution $u$ under consideration:
\begin{equation}\label{energy1}
\begin{aligned}
&E_a^{(b)}(s):=\int_{Q_{a}^{(b)}(s)}|\nabla_xu(t,x)|^{p+1}dxdt\quad
Q_{a}^{(b)}(s):=\left\{(t,x)\in(a,b)\times\Omega(s)\right\},
\\&h_a^{(b)}(s):=\sup_{a\leq t<b}\int_{\Omega(s)}|u(t,x)|^{p+1}dx\ \ \ \forall s\in(0,s_\Omega).
\end{aligned}
\end{equation}

\begin{lemma}\label{Lem1}
Let $u(t,x)$ be an arbitrary energy solution of the problem
\eqref{equat}--\eqref{initial}. Then there exist constants $k_1<\infty$, $k_2<\infty$
that don't depend on $s,a,b$ such that for arbitrary $a,b:0\leq a<b<T$ and almost all
$s\in(0,s_\Omega)$ the following inequality holds:
\begin{equation}\label{energy2}
\begin{aligned}
&h(b,s)+\frac{d_0(p+1)}pE_a^{(b)}(s)\leq h(a,s)+
\biggr(\int_a^b\int_{\partial\Omega(s)}|\nabla_xu(t,x)|^{p+1}d\sigma dt\biggr)^\frac
p{p+1}\times \\&\times\biggr[k_1\int_a^bh(t,s)dt+
k_2\big(E_a^{(b)}(s)\big)^\theta\Big(\int_a^bh(t,s)dt\Big)^{1-\theta}\biggr]^\frac1{p+1},
\quad\text{where }\theta=(p+1)^{-1}.
\end{aligned}
\end{equation}
\end{lemma}
The proof repeats the proof of lemma 6.2.2 from \cite{KSSh} with small changes (see also
lemma 3.1 from \cite{DV1}).

\begin{lemma}\label{Lem2}
Under conditions and definitions of lemma \ref{Lem1} energy functions, connected with an
arbitrary solution $u$ of the problem under consideration, satisfy the following
relations:
\begin{equation}\label{stat1}
E_a^{(b)}(s)\leq C_1\int_{\Omega(s)}|u(a,x)|^{p+1}dx+C_2(b-a)^\frac1{p+1} \biggr(-\frac
d{ds}E_a^{(b)}(s)\biggr),
\end{equation}
\begin{equation}\label{stat2}
h_a^{(b)}(s)\leq(1+\gamma)\int_{\Omega(s)}|u(a,x)|^{p+1}dx+
C_3\gamma^{-\frac1{p+1}}(b-a)^\frac1{p+1}\biggr(-\frac d{ds}E_a^{(b)}(s)\biggr)
\end{equation}
for almost all $s\in(0,s_\Omega)$ and an arbitrary $\gamma:0<\gamma<1$. Here positive
constants $C_1<\infty$, $C_2<\infty$, $C_3<\infty$ depend on known parameters of the
problem under consideration only and particulary don't depend on $\gamma$.
\end{lemma}
\begin{proof}
Using standard way (see, for example, \cite{DV1}) we deduce that the function
$E_a^{(b)}(\cdot)$ from \eqref{energy1} is differentiable almost everywhere and the
following relation holds:
\begin{equation}\label{dif1}
-\frac{d}{ds}E_a^{(b)}(s)\geq\int_a^b
\int_{\partial\Omega(s)}|\nabla_xu(t,x)|^{p+1}d\sigma dt\ \ \ \text{for a.a. }
s\in(0,s_\Omega).
\end{equation}
Then starting from relation \eqref{energy2} from lemma \ref{Lem1} and using inequality
\eqref{dif1} we deduce inequalities \eqref{stat1}, \eqref{stat2} after simple
computations that are analogous to the proof of lemma 6.2.3 from \cite{KSSh}.
\end{proof}
Now we implement some construction which is essential for our further analysis. Firstly
we introduce the sequence $\{t_i\}$, $i=1,2,...$, $t_0=0$, $t_{i}-t_{i-1}:=\Delta_i>0$
for all $i\in{N}$, $t_i\rightarrow T$ as $i\rightarrow\infty$. This sequence is defined
by the function $F(t)$ from condition \eqref{global} namely
\begin{equation}\label{sequence1}
    \Delta_i:=(p+\mu)\omega^{p+\mu}\eta^{-(p+\mu)}(i+L)^{-(p+\mu+1)},\ \ \ i=1,2,...,
\end{equation}
where parameters $\omega$, $\mu$, $p$ are from \eqref{global}, $\eta$ and $L$ will be
defined later. Firstly we have to guarantee the following equality:
\begin{equation}\label{normal1}
    \sum_{i=1}^\infty\Delta_i=T\Rightarrow
    T=(p+\mu)\omega^{p+\mu}\eta^{-(p+\mu)}\sum_{i=1}^\infty(i+L)^{-(p+\mu+1)}:=
    (p+\mu)\omega^{p+\mu}\eta^{-(p+\mu)}\sigma(L,p+\mu),
\end{equation}
which can be considered as some relation between free parameters $\eta$, $L$. Now fix
constants $\gamma>0$, $\nu>0$ and sufficiently large $L_0>0$ such that
\begin{equation}\label{theta0}
\theta_0:=(1+\gamma)\biggr(\frac{1+L_0}{L_0}\biggr)^{\frac{p+\mu+1}{p+1}}\exp\biggr[-\frac{L_0}{(1+\nu)(1+L_0))}\biggr]<1,
\end{equation}
and further we will suppose that the parameter $L$ from \eqref{normal1} satisfies the
condition:
\begin{equation}\label{28.1}
    L\geq L_0.
\end{equation}
Note that due to the monotonicity of the function $\sigma(\cdot,p+\mu)$ the sufficient
condition for $\omega$, which guarantees the fulfilment of the relation \eqref{normal1},
is as follows:
\begin{equation}\label{condomega1}
    \omega\geq\tilde\omega=\biggr(\frac{T}{(p+\mu)\sigma(L_0,p+\mu)}\biggr)^{\frac1{p+\mu}}\eta.
\end{equation}
It means that for an arbitrary $\omega$ from \eqref{condomega1} there exists
$L=L(\omega)\geq L_0$ such that the relation \eqref{normal1} is satisfied. In virtue of
the following condition
\begin{equation*}
    \int_{j-1}^j\frac{ds}{s^{p+\mu+1}}>j^{-(p+\mu+1)}>\int_j^{j+1}\frac{ds}{s^{p+\mu+1}}
\end{equation*}
the following inequalities are valid:
\begin{equation}\label{28.2}
\omega^{p+\mu}\eta^{-(p+\mu)}(j+L+1)^{-(p+\mu)}<T-t_j:=\sum_{i=j+1}^\infty\Delta_i<\omega^{p+\mu}\eta^{-(p+\mu)}(j+L)^{-(p+\mu)}.
\end{equation}
These relations lead to
\begin{equation}\label{28.3}
\exp\big(\eta(j+L)\big)\leq\exp\biggr(\omega(T-t_j)^{-\frac1{p+\mu}}\biggr) \leq
e^\eta\exp\big(\eta(j+L)\big)
\end{equation}
and moreover,
\begin{equation}\label{28.4}
 \begin{aligned}
&\Delta_j=(p+\mu)\omega^{p+\mu}\eta^{-(p+\mu)}\big((j+L)^{-1}\big)^{(p+\mu+1)}\geq
\\&\geq(p+\mu)\omega^{p+\mu}\eta^{-(p+\mu)}\omega^{-(p+\mu+1)}\eta^{p+\mu+1}(T-t_j)^{\frac{p+\mu+1}{p+\mu}}
=(p+\mu)\omega^{-1}\eta(T-t_j)^{\frac{p+\mu+1}{p+\mu}}.
 \end{aligned}
\end{equation}
Analogously,
\begin{equation}\label{28.5}
\Delta_j\leq(p+\mu)\omega^{-1}\eta(T-t_j)^{\frac{p+\mu+1}{p+\mu}}(1+\xi_j),
\end{equation}
where $\xi_j=\biggr(1-\frac1{j+L+1}\biggr)^{-(p+\mu+1)}-1\rightarrow0$ as
$j\rightarrow\infty$.

Introduce now an infinite family of energy functions connected with a solution $u(t,x)$
under consideration:
\begin{equation}\label{28.6}
    E_j(s):=E_{t_{j-1}}^{(t_j)}(s),\ \ h_j(s):=h_{t_{j-1}}^{(t_j)}(s)\ \ \ \forall s>0,\
    j=1,2,...,
\end{equation}
where $E_a^{(b)}(s)$, $h_a^{(b)}(s)$ are from \eqref{energy1}, \eqref{stat2}. Then
system \eqref{stat1}, \eqref{stat2} leads to the following infinite system of ordinary
differential inequalities (ODI):
\begin{equation}\label{28.7}
    E_j(s)\leq C_1h_{j-1}(s)+C_2\Delta_j^\frac1{p+1}
    \biggr(-\frac d{ds}E_j(s)\biggr),\ \ \ j=1,2,...,
\end{equation}
\begin{equation}\label{28.8}
    h_j(s)\leq(1+\gamma)h_{j-1}(s)+ C_3\gamma^{-\frac1{p+1}}\Delta_j^\frac1{p+1}
    \biggr(-\frac d{ds}E_j(s)\biggr),\ \ \ j=1,2,...
\end{equation}
for almost all $s\in(0,s_\Omega)$, $\forall\gamma:0<\gamma<1$. Further we will need the
following consequence of the system \eqref{28.7}, \eqref{28.8}. Namely put
$\gamma=\gamma_0=2^{-1}$ in \eqref{28.8} and add obtained inequality to \eqref{28.7}. As
a result we have:
\begin{equation}\label{28.9}
h_j(s)+E_j(s)\leq \overline{C}_1h_{j-1}(s)+\overline{C}_2\Delta_j^\frac1{p+1}
\biggr(-\frac d{ds}E_j(s)\biggr),
\end{equation}
where $\overline{C}_1=C_1+\frac32$, $\overline{C}_2=C_2+2^{p+1}C_3$.

Now we will realize detailed analysis of the asymptotic behavior of a solution of the
system \eqref{28.7}, \eqref{28.8}, \eqref{28.9}, satisfying the corresponding initial
conditions.

\begin{lemma}\label{Lem3}
Let constants $\nu>0$, $L_0>0$, $\gamma>0$ satisfy \eqref{theta0} and as a consequence
we have the following relation:
\begin{equation}\label{28.10}
    \lambda_0:=(1+\gamma)^{-1}\theta_0\biggr(\frac{L_0}{1+L_0}\biggr)^{\frac{p+\mu+1}{p+1}}
    =\exp\biggr[-\frac{L_0}{(1+\nu)(1+L_0))}\biggr]<1.
\end{equation}
Let a sequence $\{\Delta_i\}$ is defined by \eqref{sequence1}, \eqref{normal1} and let
relations \eqref{28.1}, \eqref{condomega1} hold. Let some infinite sequence of
nonnegative nonincreasing absolutely continuous functions $E_j(s)$ satisfies ODI
\eqref{28.7} for almost all $s\in(0,s_\Omega)$, where a sequence of nonnegative
nonincreasing functions $h_j(s)$ satisfies inequality \eqref{28.8}. Let the following
initial condition holds:
\begin{equation}\label{28.12}
    E_j(0)\leq G\exp\big(\eta(j+L+1)\big)\ \ \forall j\in{N},
    \ \eta>\ln\lambda_0^{-1},\ G=\text{const}\geq1.
\end{equation}
Then functions $\{E_j(s)\}$ satisfy the following estimate:
\begin{equation}\label{17.1}
\begin{aligned}
&E_j(s)\leq K_5G\exp\big(\ln \lambda_0^{-1}(j+L)\big)U_\omega(s),\ \ \ \forall j\in{N},\
\forall s\in(0,s_\Omega),
\\&\text{where }U_\omega(s)=\exp\Big(K_6\gamma^{-\frac1\mu}\eta^{-\frac{p+\mu}\mu}\big(\eta-\ln\lambda_0^{-1}\big)^\frac{p+\mu+1}\mu
\omega^\frac{p+\mu}\mu s^{-\frac{p+1}\mu}\Big),
\end{aligned}
\end{equation}
and constants $K_5,K_6<\infty$ do not depend on $G,\eta,\omega,\gamma$;
$s':=\min(s_\Omega,s_\omega)$, $s_\omega=K_7\omega^\frac{p+\mu}{p+1}$, $K_7>0$ does not
depend on $\omega,\,G$.
\end{lemma}

\begin{proof}
Let us introduce a sequence of functions $\{A_i(s)\}$, $\{H_i(s)\}$ connected with
$\{E_i(s)\}$, $\{h_i(s)\}$, namely
\begin{equation}\label{29.1}
A_j(s):=\lambda_0^{j+L}E_j(s),\ \ H_j(s):=\lambda_0^{j+L}h_j(s)\ \ \ \forall j\in{N},\
H_0(s)=\lambda_0^Lh_0(s),
\end{equation}
where $\lambda_0$ is from \eqref{28.10}. With respect to these functions the system
\eqref{28.7}, \eqref{28.8} yields:
\begin{equation}\label{29.3}
    A_j(s)\leq C_1\lambda_0H_{j-1}(s)+C_2\Delta_j^\frac1{p+1}
    \biggr(-\frac d{ds}A_j(s)\biggr)\ \ \ j=1,2,...,
\end{equation}
\begin{equation}\label{29.4}
    H_j(s)\leq\lambda_1H_{j-1}(s)+ C_3\gamma^{-\frac1{p+1}}\Delta_j^\frac1{p+1}
    \biggr(-\frac d{ds}A_j(s)\biggr)\ \ \ \forall j\in{N},\text{ for a.a. }s\in(0,s_\Omega),
\end{equation}
where $\lambda_1:=(1+\gamma)\lambda_0$. It follows from \eqref{28.10} that
$\lambda_1=\theta_0\Big(\frac{L_0}{1+L_0}\Big)^\frac{p+\mu+1}{p+1}<1$. Now we estimate
the first term in the right hand side of \eqref{29.3} using \eqref{29.4} with an index
$j$ instead of $(j-1)$. After $j$ such iterations we arrive at the following relation:
\begin{equation}\label{29.5}
A_j(s)\leq C_1(1+\gamma)^{-1}\lambda_1^jH_0(s)+
\overline{C}_3\gamma^{-\frac1{p+1}}\Delta_j^\frac1{p+1}
\sum_{i=1}^j\lambda_1^{j-i}\biggr(\frac{\Delta_i}{\Delta_j}\biggr)^\frac1{p+1}\big(-A'_i(s)\big),
\end{equation}
where $\overline{C}_3=\max\{C_2\gamma^\frac1{p+1},\lambda_0C_1C_3\}$. Now with respect
to following energy functions:
\begin{equation}\label{29.6}
    U_j(s):=\sum_{i=1}^j\lambda_1^{j-i}\biggr(\frac{\Delta_i}{\Delta_j}\biggr)^\frac1{p+1}A_i(s),
    \ \ \ j=1,2,...,
\end{equation}
the relation \eqref{29.5} leads to the following inequality:
\begin{equation}\label{29.7}
\begin{aligned}
&U_j(s)\leq C_1(1+\gamma)^{-1}\lambda_1^jH_0(s)+\theta_jU_{j-1}(s)+
\overline{C}_3\gamma^{-\frac1{p+1}}\Delta_j^\frac1{p+1}\big(-U'_j(s)\big),
\\&j=1,2,...,\ H_0(s)=\lambda_0^{L}h_0(s).
\end{aligned}
\end{equation}
Due to \eqref{sequence1}, \eqref{28.1}, \eqref{28.10} and \eqref{theta0} we can get:
\begin{equation}\label{30.1}
\theta_j=\lambda_1\biggr(\frac{\Delta_{j-1}}{\Delta_j}\biggr)^\frac1{p+1}=
(1+\gamma)\lambda_0\biggr(\frac{j+L}{j+L-1}\biggr)^\frac{p+\mu+1}{p+1}\leq
(1+\gamma)\lambda_0\biggr(\frac{1+L_0}{L_0}\biggr)^\frac{p+\mu+1}{p+1}=\theta_0<1.
\end{equation}
From \eqref{28.12}, \eqref{29.1} and \eqref{28.4} we obtain the following estimate:
\begin{equation}\label{30.3}
 \begin{aligned}
 &U_j(0)=\sum_{i=1}^j\lambda_1^{j-i}\biggr(\frac{\Delta_i}{\Delta_j}\biggr)^\frac1{p+1}\lambda_0^{i+L}Ge^\eta\exp\big(\eta(j+L)\big)\leq
\\&\leq\sum_{i=1}^j((1+\gamma)\lambda_0)^{j-i}\biggr(\frac{j+L}{i+L}\biggr)^\frac{p+\mu+1}{p+1}\lambda_0^{i+L}Ge^\eta\exp\big(\eta(j+L)\big)\leq
 \\&\leq Ge^\eta\sum_{i=1}^j\theta_0^{j-i}\exp\big[(\eta-\ln\lambda_0^{-1})(i+L)\big]\leq
\\&\leq Ge^\eta(1-\theta_0)^{-1}\exp\big[(\eta-\ln\lambda_0^{-1})(j+L)\big]\ \ \
 \forall j\in{N}.
 \end{aligned}
\end{equation}
Additionally it follows from \eqref{sequence1} that
$(j+L)=(p+\mu)^\frac1{p+\mu+1}\omega^\frac{p+\mu}{p+\mu+1}\eta^{-\frac{p+\mu}{p+\mu+1}}\Delta_j^{-\frac1{p+\mu+1}}$
$\forall\,j\in{N}$. Therefore estimate \eqref{30.3} yields:
\begin{equation}\label{30.4}
 U_j(0)\leq Ge^\eta(1-\theta_0)^{-1}
 \exp\big[(p+\mu)^\frac1{p+\mu+1}\omega^\frac{p+\mu}{p+\mu+1}\eta^{-\frac{p+\mu}{p+\mu+1}}(\eta-\ln\lambda_0^{-1})\Delta_j^{-\frac1{p+\mu+1}}\big].
\end{equation}
Since $C_1(1+\gamma)^{-1}\lambda_1^j\leq C_1(1+\gamma)^{-1}\lambda_1=C_1\lambda_0$ it
follows from \eqref{29.7}, \eqref{30.4}:
\begin{equation}\label{30.5}
\begin{aligned}
&\overline{U}_j(s):=U_j(s)-\frac{C_1\lambda_0}{1-\theta_0}H_0(s)\leq\theta_0\overline{U}_{j-1}(s)+
\overline{C}_3\gamma^{-\frac1{p+1}}\Delta_j^\frac1{p+1}\big(-\overline{U}'_j(s)\big),
\\&\overline{U}_j(0)\leq\frac{Ge^\eta}{1-\theta_0}
 \exp\big[(p+\mu)^\frac1{p+\mu+1}\omega^\frac{p+\mu}{p+\mu+1}\eta^{-\frac{p+\mu}{p+\mu+1}}(\eta-\ln\lambda_0^{-1})\Delta_j^{-\frac1{p+\mu+1}}\big]
\ \ \forall j\in{N}.
\end{aligned}
\end{equation}
Now we rewrite the last system in the form:
\begin{equation}\label{1.1}
\begin{aligned}
&\tilde{U}_j(s):=G^{-1}e^{-\eta}(1-\theta_0)\overline{U}_j(s)\leq
\theta_0\tilde{U}_{j-1}(s)+(1-\theta_0)m_j\big(-\tilde{U}'_j(s)\big),
\\&\tilde{U}_j(0)\leq\exp{M_j}\ \ \
\forall j\in{N},
\end{aligned}
\end{equation}
where $m_j=(1-\theta_0)^{-1}\overline{C}_3\gamma^{-\frac1{p+1}}\Delta_j^\frac1{p+1}$,
$M_j=(p+\mu)^\frac1{p+\mu+1}\eta^{-\frac{p+\mu}{p+\mu+1}}\big(\eta-\ln\lambda_0^{-1}\big)
\omega^\frac{p+\mu}{p+\mu+1}\Delta_j^{-\frac1{p+\mu+1}}$. Due to lemmas 9.2.7, 9.2.8
from \cite{KSSh} it follows from \eqref{1.1} a uniform upper estimate:
\begin{equation}\label{1.2}
\tilde{U}_j(s)\leq U(s):=\exp\big(B_1\omega^\frac{p+\mu}\mu s^{-\frac{p+1}\mu}\big)\ \ \
\forall\,s\in(0,s_\Omega),\ \forall j\in{N},
\end{equation}
where $B_1=C_4\gamma^{-\frac1\mu}\eta^{-\frac{p+\mu}\mu}
\big(\eta-\ln\lambda_0^{-1}\big)^\frac{p+\mu+1}\mu$,
$C_4=\frac{\mu(p+\mu)^\frac1{\mu}(p+1)^\frac{p+1}\mu\overline{C}_3^\frac{p+1}\mu}
{(p+\mu+1)^\frac{p+\mu+1}{\mu}(1-\theta_0)^{\frac{p+1}\mu}}$. Therefore it follows from
definitions \eqref{29.1}, \eqref{28.12} that:
\begin{equation}\label{1.3}
\begin{aligned}
&E_j(s)=\lambda_0^{-(j+L)}A_j(s)\leq\lambda_0^{-(j+L)}U_j(s)\leq
\lambda_0^{-(j+L)}\biggr(\frac{\tilde{U}_j(s)Ge^\eta}{1-\theta_0}+
\frac{C_1\lambda_0H_0(s)}{1-\theta_0}\biggr)\leq
\\&\leq(1-\theta_0)^{-1}\lambda_0^{-(j+L)}
\Big(Ge^\eta\exp\big[B_1\omega^\frac{p+\mu}\mu
s^{-\frac{p+1}\mu}\big]+C_1\lambda_0H_0(s)\Big)\leq
\\&\leq C_5G\lambda_0^{-(j+L)}
\exp\big[B_1\omega^\frac{p+\mu}\mu s^{-\frac{p+1}\mu}\big]\ \ \ \forall s:0<s<s',\
C_5=2(1-\theta_0)^{-1}e^\eta
\end{aligned}
\end{equation}
where $s'=\min(s_\Omega,s_\omega)$ and $s_\omega$ is defined by the following relation:
\begin{equation}\label{1.4}
B_1\omega^\frac{p+\mu}\mu s^{-\frac{p+1}\mu}\geq B_2:=\ln C_1+\ln\lambda_0+\ln H_0(0)\ \
\ \ \forall s:0<s\leq
s_\omega=B_1^\frac\mu{p+1}B_2^{-\frac\mu{p+1}}\omega^\frac{p+\mu}{p+1}.
\end{equation}
Thus, the estimate \eqref{17.1} is proved with $K_5=C_5$, $K_6=C_4$,
$K_7=B_1^\frac\mu{p+1}B_2^{-\frac\mu{p+1}}$
\end{proof}

\section{Proof of Theorem \ref{Th.1}: rough estimate of solution near the blow-up time}

Let $u$ be a solution of the problem \eqref{equat}--\eqref{initial} under consideration
and families of subdomains $\Omega(s)$, $Q(s)$ are from \eqref{fam1}. Let $\{E_j(s)\}$,
$\{h_j(s)\}$ be families of energy functions \eqref{28.6} that connect with $u$ and
correspond to the family $\{\Delta_j\}$ from \eqref{sequence1}. Let the parameter $\eta$
from definition \eqref{sequence1} satisfies the inequality:
\begin{equation}\label{cond23.11}
\eta>\ln\lambda_0^{-1}=\frac{L_0}{(1+\nu)(1+L_0)},
\end{equation}
where $L_0,\nu$ are from \eqref{28.10}, and let the inequality \eqref{condomega1} holds.
Then the system \eqref{28.7}, \eqref{28.8} is satisfied. Moreover due to the condition
\eqref{global} and the properties \eqref{28.3} we have:
\begin{equation*}
E(t_j)\leq\exp\Big(\omega(T-t_j)^{-\frac1{p+\mu}}\Big) \leq\exp(\eta(j+L+1))\ \ \
\forall j\in{N}
\end{equation*}
and therefore
\begin{equation}\label{18.1}
E_j(0)\leq\exp(\eta(j+L+1))\ \ \ \forall j\in{N}.
\end{equation}
Consequently in virtue of lemma \ref{Lem3} energy functions $\{E_j(s)\}$, $j=1,2,...$,
satisfy estimate \eqref{17.1} with $G=1$, $K_5=C_5$ from \eqref{1.3}, $K_6=C_4$ from
\eqref{1.2}, namely
\begin{equation}\label{18.2}
E_j(s)\leq C_5\exp(\ln\lambda_0^{-1}(j+L))U_\omega(s)\ \ \ \forall j\in{N},\ \forall
s:0<s<s'=\min(s_\Omega,s_\omega).
\end{equation}
Let us fix an arbitrary value $s_1:0<s_1<s'$ and write the inequality \eqref{18.2} in
the following form:
\begin{equation}\label{18.3}
E_j(s_1)\leq C_5G_1\lambda_0^{-(j+L)}\ \ \ \forall j\in{N},\ \ G_1:=U_\omega(s_1).
\end{equation}
Summing the estimates \eqref{18.3} we get
\begin{equation}\label{18.4}
E(t_j,s_1):=\int_0^{t_j}\int_{\Omega(s_1)}|\nabla_xu(t,x)|^{p+1}dxdt\leq
C_6G_1\lambda_0^{-(j+L)}\ \ \ \forall j\in{N},\ C_6=C_5(1-\lambda_0)^{-1}.
\end{equation}
Due to \eqref{28.3} and \eqref{28.10} the estimate \eqref{18.4} yields:
\begin{equation}\label{18.5}
\begin{aligned}
&E(t_j,s_1)\leq
C_6G_1\exp\Big(\ln\lambda_0^{-1}\eta^{-1}\omega(T-t_j)^{-\frac1{p+\mu}}\Big)=
\\&=C_6G_1\exp\biggr(\frac{L_0}{(1+\nu)(1+L_0)\eta}\omega(T-t_j)^{-\frac1{p+\mu}}\biggr)=
C_6G_1\exp\Big(\omega_1(T-t_j)^{-\frac1{p+\mu}}\Big),
\end{aligned}
\end{equation}
where due to \eqref{cond23.11} we get:
\begin{equation}\label{18.6}
\omega_1=\zeta\omega,\ \ \ \zeta=\eta^{-1}\ln\lambda_0^{-1}<1.
\end{equation}
Finally it follows from \eqref{18.5} that in virtue of \eqref{18.3} the following
inequality holds:
\begin{equation}\label{18.7}
E(t,s_1)\leq C_7G_1\exp\Big(\omega_1(T-t)^{-\frac1{p+\mu}}\Big)\ \ \ \forall t<T,\
C_7=C_6\lambda_0^{-1}.
\end{equation}
Estimate \eqref{18.7} is the final result of the first round of computations. Now we
begin the second round. Suppose, that the constant $\omega_1$ from \eqref{18.6}
satisfies condition \eqref{condomega1}, namely:
\begin{equation}\label{18.8}
\omega_1\geq\tilde{\omega}=\biggr(\frac{T}{(p+\mu)\sigma(L_0,p+\mu)}\biggr)^\frac1{p+\mu}\eta.
\end{equation}
Then introduce a new sequence of shifts $\{\Delta_j^{(1)}\}$:
\begin{equation}\label{18.9}
    \Delta_j^{(1)}:=(p+\mu)\omega_1^{p+\mu}\eta^{-(p+\mu)}(j+L)^{-(p+\mu+1)},\ \ \
j=1,2,...\ .
\end{equation}
It is clear that the analog of relations \eqref{28.3} holds, namely:
\begin{equation}\label{18.10}
\exp\big(\eta(j+L)\big)\leq\exp\biggr(\omega_1(T-t_j^{(1)})^{-\frac1{p+\mu}}\biggr) \leq
e^\eta\exp\big(\eta(j+L)\big),
\end{equation}
and as consequence
\begin{equation}\label{18.11}
\Delta_j^{(1)}\geq(p+\mu)\omega_1^{-1}\eta(T-t_j^{(1)})^{\frac{p+\mu+1}{p+\mu}}\ \ \
\forall j\in{N},
\quad\Delta_j^{(1)}\leq(p+\mu)\omega_1^{-1}\eta(T-t_j^{(1)})^{\frac{p+\mu+1}{p+\mu}}(1+\xi_j),
\end{equation}
where $\xi_j\rightarrow0$ as $j\rightarrow\infty$. Now introduce new energy functions:
\begin{equation*}
E_j^{(1)}(s):=E_{t_{j-1}^{(1)}}^{\big(t_j^{(1)}\big)}(s),\ \
h_j^{(1)}(s):=h_{t_{j-1}^{(1)}}^{\big(t_j^{(1)}\big)}(s)\ \ \ \forall s>s_1.
\end{equation*}
It is obvious that these functions satisfy the following analog of \eqref{28.7},
\eqref{28.8}.
\begin{equation}\label{18.12}
    E_j^{(1)}(s)\leq C_1h_{j-1}^{(1)}(s)+C_2\big(\Delta_j^{(1)}\big)^\frac1{p+1}
    \biggr(-\frac d{ds}E_j^{(1)}(s)\biggr),
\end{equation}
\begin{equation}\label{18.13}
    h_j^{(1)}(s)\leq(1+\gamma)h_{j-1}^{(1)}(s)+ C_3\gamma^{-\frac1{p+1}}\big(\Delta_j^{(1)}\big)^\frac1{p+1}
    \biggr(-\frac d{ds}E_j^{(1)}(s)\biggr)
\end{equation}
for almost all $s\in(s_1,s_\Omega)$ and all $j\in{N}$. In virtue of \eqref{18.10} the
estimate \eqref{18.7} yields:
\begin{equation}\label{18.14}
E_j^{(1)}(s_1)\leq C_7G_1\exp\big(\eta(j+L+1)\big)\ \ \ \forall j\in{N}.
\end{equation}
Then due to lemma \ref{Lem3} it follows from \eqref{18.12}--\eqref{18.14} the following
estimate:
\begin{equation}\label{18.15}
E_j^{(1)}(s)\leq
C_5C_7G_1\exp\big(\ln\lambda_0^{-1}(j+L)\big)U_{\omega_1}(s-s_1)\quad\forall\,s\in(s_1,s'),\,\forall\,j\in{N}.
\end{equation}
Particulary for $s=s_1+s_2<s'$ we have:
\begin{equation}\label{18.16}
E_j^{(1)}(s_1+s_2)\leq
 C_5C_7G_1\exp\big(\ln\lambda_0^{-1}(j+L)\big)U_{\omega_1}(s_2):=
 C_5C_7G_1G_2\lambda_0^{-(j+L)}\ \ \ \forall j\in{N},
\end{equation}
where $G_2=U_{\omega_1}(s_2)$. Summing estimates \eqref{18.16} we get:
\begin{equation}\label{18.17}
E(t_j^{(1)},s_1+s_2)\leq C_6C_7G_1G_2\lambda_0^{-(j+L)}\ \ \ \forall j\in{N}.
\end{equation}
Due to \eqref{18.10} the estimate \eqref{18.17} yields:
\begin{equation}\label{18.18}
\begin{aligned}
&E(t,s_1+s_2)\leq
C_7^2G_1G_2\exp\big(\ln\lambda_0^{-1}\eta^{-1}\omega_1(T-t)^{-\frac1{p+\mu}}\big)=
C_7^2G_1G_2\exp\big(\omega_2(T-t)^{-\frac1{p+\mu}}\big)
\\&\forall t<T,\omega_2=\zeta\omega_1,\text{ where }
\zeta=\ln\lambda_0^{-1}\eta^{-1}<1.
\end{aligned}
\end{equation}
Estimate \eqref{18.18} is the final estimate of the second round of computations. It is
clear that we can realize $l$ such a rounds, where $l$ is defined by
\begin{equation}\label{18.19}
\omega_l\geq\tilde{\omega}, \ \ \ \tilde{\omega}\text{ is from } \eqref{condomega1}, \
\omega_{l+1}<\tilde{\omega}.
\end{equation}
As result we obtain:
\begin{equation}\label{18.20}
E\Big(t,\sum_{i=1}^ls_i\Big)\leq
C_7^l\prod_{i=1}^lG_i\exp\big(\omega_l(T-t)^{-\frac1{p+\mu}}\big) \ \ \ \forall
t<T,\,\sum_{i=1}^ls_i\leq s'.
\end{equation}
Moreover we have:
\begin{equation}\label{18.21}
\begin{aligned}
&\prod_{i=1}^lG_i=\prod_{i=1}^lU_{\omega_{i-1}}(s_i)=\exp\biggr[C_4\gamma^{-\frac1\mu}\eta^{-\frac{p+\mu}\mu}
d^\frac{p+\mu+1}\mu\sum_{i=1}^l\omega_{i-1}^\frac{p+\mu}\mu s_i^{-\frac{p+1}\mu}\biggr]=
\\&=\exp\biggr[C_4\gamma^{-\frac1\mu}\eta^{-\frac{p+\mu}\mu}
d^\frac{p+\mu+1}\mu\omega^\frac{p+\mu}\mu s_1^{-\frac{p+1}\mu}
\sum_{i=1}^l\biggr(\frac{\omega_{i-1}}\omega\biggr)^\frac{p+\mu}\mu\biggr(\frac{s_1}{s_i}\biggr)^\frac{p+1}\mu
\biggr],
\\&\omega_0:=\omega,d=\eta-\ln\lambda_0^{-1}>0.
\end{aligned}
\end{equation}
It is clear that $\frac{\omega_{i-1}}\omega=\zeta^{i-1}$, where
$\zeta=\frac{\ln\lambda_0^{-1}}{\ln\lambda_0^{-1}+d}<1$. As for $s_i$, we define it by
the following relation:
\begin{equation}\label{18.22}
s_i:=s_1\rho^{i-1}\ \ \forall i\geq1,
\end{equation}
where $\rho=const<1$ will be defined later. Then
\begin{equation}\label{18.23}
\biggr(\frac{\omega_{i-1}}\omega\biggr)^\frac{p+\mu}\mu\biggr(\frac{s_1}{s_i}\biggr)^\frac{p+1}\mu=\zeta^\frac{(i-1)(p+\mu)}\mu
\rho^{-\frac{(i-1)(p+1)}\mu}=\biggr(\frac{\zeta^{p+\mu}}{\rho^{p+1}}\biggr)^\frac{i-1}\mu.
\end{equation}
Let us define now $\rho$ as follows:
\begin{equation}\label{18.24}
\rho^{p+1}=\frac{1+\zeta^{p+\mu}}2\quad\Rightarrow\quad\rho:=\left(2^{-1}(1+\zeta^{p+\mu})\right)^\frac1{p+1}.
\end{equation}
Then
\begin{equation}\label{18.25}
\biggr(\frac{\zeta^{p+\mu}}{\rho^{p+1}}\biggr)^\frac1\mu=
\biggr(\frac{2\zeta^{p+\mu}}{1+\zeta^{p+\mu}}\biggr)^\frac1\mu:={\ae}=\text{const}<1
\end{equation}
Consequently,
\begin{equation}\label{18.26}
\sum_{i=1}^l\biggr(\frac{\omega_{i-1}}\omega\biggr)^\frac{p+\mu}\mu\biggr(\frac{s_1}{s_i}\biggr)^\frac{p+1}\mu=
\sum_{i=1}^l{\ae}^{i-1}<\frac1{1-{\ae}}.
\end{equation}
Additionally,
\begin{equation}\label{18.27}
\sum_{i=1}^ls_i=s_1\sum_{i=1}^l\rho^{i-1}<s_1\frac1{1-\rho}:=S^{(\infty)}.
\end{equation}
Therefore it follows from \eqref{18.21}:
\begin{equation}\label{18.30}
\prod_{i=1}^lG_i\leq\exp\Big[C_4\gamma^{-\frac1\mu}\eta^{-\frac{p+\mu}\mu}
d^\frac{p+\mu+1}\mu\omega^\frac{p+\mu}\mu
s_1^{-\frac{p+1}\mu}(1-{\ae})^{-1}\Big]:=G^{(\infty)}(s_1).
\end{equation}
Now in virtue of \eqref{18.27}, \eqref{18.30} estimate \eqref{18.20} yields for an
arbitrary $s_1:(1-\rho)^{-1}s_1<s'$ the following inequality:
\begin{equation}\label{18.28}
E\Big(t,\frac{s_1}{1-\rho}\Big)\leq
C_7^l\exp\big(\omega_l(T-t)^{-\frac1{p+\mu}}\big)\exp\Big[C_4\gamma^{-\frac1\mu}\eta^{-\frac{p+\mu}\mu}
d^\frac{p+\mu+1}\mu(1-{\ae})^{-1}\omega^\frac{p+\mu}\mu s_1^{-\frac{p+1}\mu}\Big],
\end{equation}
which implies
\begin{equation}\label{18.29}
E(t,s)\leq
C_7^l\exp\big(\omega_l(T-t)^{-\frac1{p+\mu}}\big)\exp[C_8\omega^{\frac{p+\mu}\mu}s^{-\frac{p+1}\mu}]\
\ \ \forall t<T,\ \forall s<(1-\rho)s',
\end{equation}
where $C_8=C_4\gamma^{-\frac1\mu}\eta^{-\frac{p+\mu}\mu}
d^\frac{p+\mu+1}\mu(1-{\ae})^{-1}(1-\rho)^{-\frac{p+1}\mu}$, $d=\eta-\ln\lambda_0^{-1}$.
Of course, in the case when the initial upper estimate \eqref{condomega1} is not
satisfied we get $l=0$ and we must begin the proof of theorem \ref{Th.1} directly from
the next step. For this next step of the proof we need to have an upper estimate of
functions $h_j(s)$ for $s>0$, $j\in{N}$ that is analogous to the estimate \eqref{18.2}
for $E_j(s)$. So we return to the relation \eqref{28.9}. Due to the absolute continuity
of $h_j(s)$ and lack of its increase it follows from \eqref{28.9} the relation:
\begin{equation}\label{25.1}
\widetilde{E}_j(s):=h_j(s)+E_j(s)\leq
\overline{C}_1h_{j-1}(s)+\overline{C}_2\Delta_j^\frac1{p+1}
    \biggr(-\frac d{ds}\widetilde{E}_j(s)\biggr)\ \ \ \forall j\in{N}
\end{equation}
for almost all $s\in(0,s_\Omega)$. Due to condition \eqref{global} $\widetilde{E}_j(s)$
satisfies the following analog of the initial condition \eqref{18.1}
\begin{equation}\label{25.2}
\widetilde{E}_j(0)\leq\exp(\eta(j+L+1))\ \ \ \forall j\in{N}.
\end{equation}
It is clear that the following relation that is analogous to \eqref{28.8} holds:
\begin{equation}\label{25.3}
    h_j(s)\leq(1+\gamma)h_{j-1}(s)+ C_3\gamma^{-\frac1{p+1}}\Delta_j^\frac1{p+1}
    \biggr(-\frac d{ds}\widetilde{E}_j(s)\biggr)\ \ \ \forall j\in{N}.
\end{equation}
Therefore repeating all realized steps in the proof of theorem \ref{Th.1} and using
relations \eqref{25.1}, \eqref{25.3}, \eqref{25.2} instead of \eqref{28.7},
\eqref{28.8}, \eqref{18.1}, we obtain the analog of the estimate \eqref{18.20}:
\begin{equation}\label{25.4}
\widetilde{E}\Big(t,\sum_{i=1}^ls_i\Big):=\widetilde{h}\Big(t,\sum_{i=1}^ls_i\Big)+E\Big(t,\sum_{i=1}^ls_i\Big)
\leq C_7^l\prod_{i=1}^lG_i\exp\big(\omega_l(T-t)^{-\frac1{p+\mu}}\big),
\end{equation}
where
\begin{equation*}
\widetilde{h}(t,s):=\sup_{\tau\in(0,t)}\int_{\Omega(s)}|u(\tau,x)|^{p+1}dx.
\end{equation*}
Thus in virtue of relations \eqref{18.19} we can not use estimate \eqref{18.20} (or
\eqref{25.4}) as the initial condition for $(l+1)$-th circle of the iterative estimation
directly. Therefore we will implement some additional trick. Namely, due to
\eqref{18.19} we have
\begin{equation*}
\omega_{l+1}=\omega\zeta^{l+1}<\tilde\omega:=\biggr(\frac{T}{(p+\mu)\sigma(L_0,p+\mu)}\biggr)^{\frac1{p+\mu}}\eta.
\end{equation*}
Therefore we can define a value $t_0^{(l)}>0$ by the relation:
\begin{equation}\label{25.5}
\begin{aligned}
&\omega_{l+1}=\omega\zeta^{l+1}=\biggr(\frac{T-t_0^{(l)}}{(p+\mu)\sigma(L_0,p+\mu)}\biggr)^{\frac1{p+\mu}}\eta\
\ \Leftrightarrow
\\&\omega_l\big(T-t_0^{(l)}\big)^{-\frac1{p+\mu}}=
\frac{\eta}{\zeta\big((p+\mu)\sigma(L_0,p+\mu)\big)^{\frac1{p+\mu}}}:=\Lambda_1.
\end{aligned}
\end{equation}
Now estimate \eqref{25.4} can be written as follows:
\begin{equation}\label{25.6}
\widetilde{E}\big(t,S^{(l)}\big)\leq\begin{cases}
C_7^l\prod_{i=1}^lG_i\exp\big(\omega_l(T-t)^{-\frac1{p+\mu}}\big),&\forall t\in(t_0^{(l)},T);\\
C_7^l\prod_{i=1}^lG_i\exp\Lambda_1,&\forall t\in(0,t_0^{(l)}] \end{cases}
\end{equation}
where $S^{(l)}=\sum_{i=1}^ls_i$. So, estimate \eqref{25.6} is the final estimate of our
solution $u(t,x)$ in the domain $\{(t,x):0<t\leq t_0^{(l)},|x|>S^{(l)}\}$. For the
consideration of a solution $u$ in the domain $\{(t,x):t_0^{(l)}<t<T,|x|>S^{(l)}\}$ we
introduce new energy functions:
\begin{equation}\label{25.7}
\widetilde{h}^{(l)}(t,s):=\sup_{t_0^{(l)}<\tau<t}\int_{\Omega(s)}|u(\tau,x)|^{p+1}dx,\
E^{(l)}(t,s):=\int_{t_0^{(l)}}^t\int_{\Omega(s)}|\nabla_xu(\tau,x)|^{p+1}dxd\tau
\end{equation}
for all $s\geq S^{(l)}$ and $t>t_0^{(l)}$. Now we introduce new shifts
$\big\{\Delta_j^{(l)}\big\}$, $j=1,2,...$ that is similar to \eqref{sequence1}, namely
\begin{equation}\label{25.8}
   \Delta_j^{(l)}:=(p+\mu)\omega_l^{p+\mu}\eta^{-(p+\mu)}(j+L)^{-(p+\mu+1)},\ \ \ j=1,2,...,
\end{equation}
\begin{equation*}
t_j^{(l)}=t_{j-1}^{(l)}+\Delta_j^{(l)},\ \ \ j=1,2,...,\ \ t_0^{(l)} \text{ is from }
\eqref{25.5}.
\end{equation*}
Now condition \eqref{normal1} has the form:
\begin{equation}\label{25.9}
\sum_{i=1}^\infty\Delta_i^{(l)}=T-t_0^{(l)}=
(p+\mu)\eta^{-(p+\mu)}\omega_l^{(p+\mu)}\sigma(L,p+\mu).
\end{equation}
As is easy to see the definition \eqref{25.5} of $t_0^{(l)}$ guarantees the validity of
\eqref{25.9} with some $L>L_0$ defined by the equation:
\begin{equation}\label{25.10}
\zeta\sigma(L_0,p+\mu)^\frac1{p+\mu}=\sigma(L,p+\mu)^\frac1{p+\mu}.
\end{equation}
Introduce now energy functions $h_j^{(l)}(s)$, $E_j^{(l)}(s)$, $j=1,2,...$, connected
with shifts $\big\{\Delta_j^{(l)}\big\}$ from \eqref{25.8}. It is clear that these
functions satisfy ODI similar to \eqref{28.7}, \eqref{28.8}, \eqref{28.9}, namely:
\begin{equation*}
    E_j^{(l)}(s)\leq C_1h_{j-1}^{(l)}(s)+C_2\big(\Delta_j^{(l)}\big)^\frac1{p+1}
    \biggr(-\frac d{ds}E_j^{(l)}(s)\biggr)\ \ \ \forall\,s\in(S^{(l)},s'),
\end{equation*}
\begin{equation*}
    h_j^{(l)}(s)\leq(1+\gamma)h_{j-1}^{(l)}(s)+ C_3\gamma^{-\frac1{p+1}}\big(\Delta_j^{(l)}\big)^\frac1{p+1}
    \biggr(-\frac d{ds}E_j^{(l)}(s)\biggr)\quad\forall\,s\in(S^{(l)},s'),
\end{equation*}
\begin{equation}\label{26.1}
h_j^{(l)}(s)+E_j^{(l)}(s)\leq
\overline{C}_1h_{j-1}^{(l)}(s)+\overline{C}_2\big(\Delta_j^{(l)}\big)^\frac1{p+1}
    \biggr(-\frac d{ds}E_j^{(l)}(s)\biggr)\quad\forall\,j\in{N},
\end{equation}
Moreover, it is easy to see that due to \eqref{25.6} following initial conditions hold:
\begin{equation}\label{26.2}
h_j^{(l)}(S^{(l)})+E_j^{(l)}(S^{(l)})\leq C_7^lG^{(l)}\exp\big(\eta(j+L+1)\big)\ \ \
\forall j\in{N},
\end{equation}
where $G^{(l)}=\prod_{i=1}^lG_i$. Applying lemma \ref{Lem3} to the system \eqref{26.1},
\eqref{26.2} we obtain the estimate:
\begin{equation*}
E_j^{(l)}(S^{(l)}+s)\leq
 C_5C_7^lG^{(l)}\exp\big(\ln\lambda_0^{-1}(j+L)\big)U_{\omega_l}(s)\ \ \ \forall
j\in{N},\ \ \forall\,s:s+S^{(l)}<s',
\end{equation*}
which yields:
\begin{equation}\label{26.3}
E_j^{(l)}(S^{(l+1)})\leq C_5C_7^lG^{(l+1)}\exp\big(\ln\lambda_0^{-1}(j+L)\big).
\end{equation}
Summing these estimates we get:
\begin{equation*}
E(t_i^{(l)},S^{(l+1)})-E(t_0^{(l)},S^{(l+1)})\leq
C_6C_7^lG^{(l+1)}\exp\big(\ln\lambda_0^{-1}(i+L)\big)\ \ \ \forall i\in{N},
\end{equation*}
and using \eqref{25.8} we derive analogously to \eqref{18.17}, \eqref{18.18}:
\begin{equation}\label{26.4}
E\big(t,S^{(l+1)}\big)\leq\begin{cases}
C_7^{l+1}G^{(l+1)}\exp\big(\omega_{l+1}(T-t)^{-\frac1{p+\mu}}\big)+E(t_0^{(l)},S^{(l+1)}),&\forall t\in(t_0^{(l)},T);\\
E(t_0^{(l)},S^{(l+1)}),& \forall t\in(0,t_0^{(l)}]. \end{cases}
\end{equation}
Using \eqref{25.6} we estimate the term $E(t_0^{(l)},S^{(l+1)})$:
\begin{equation*}
E(t_0^{(l)},S^{(l+1)})\leq E(t_0^{(l)},S^{(l)})\leq
C_7^{l}G^{(l)}\exp\big(\omega_{l}(T-t_0^{(l)})^{-\frac1{p+\mu}}\big).
\end{equation*}
Therefore
\begin{equation*}
C_7^{l+1}G^{(l+1)}\exp\big(\omega_{l+1}(T-t)^{-\frac1{p+\mu}}\big)+
E(t_0^{(l)},S^{(l+1)})\leq
C_7^{l+1}G^{(l+1)}\exp\big(\omega_{l+1}(T-t)^{-\frac1{p+\mu}}\big)(1+\varepsilon_l(t)),
\end{equation*}
where
\begin{equation}\label{26.5}
\begin{aligned}
\varepsilon_l(t)&=C_7^{-1}\big(U_{\omega_l}(s_{l+1})\big)^{-1}
\exp\Big[\omega_{l}(T-t_0^{(l)})^{-\frac1{p+\mu}}-\omega_{l+1}(T-t)^{-\frac1{p+\mu}}\Big]\leq
\\&C_7^{-1}\exp\Big[(\omega_{l}-\omega_{l+1})(T-t_0^{(l)})^{-\frac1{p+\mu}}\Big]=
C_7^{-1}\exp\Big[(1-\zeta)\omega_{l}(T-t_0^{(l)})^{-\frac1{p+\mu}}\Big]=
\\&\text{(due to \eqref{25.5})}= C_7^{-1}\exp\Big[(1-\zeta)\Lambda_1\Big]=: C_7^{-1}c_\zeta\ \ \
\forall t\in(t_0^{(l)},T).
\end{aligned}
\end{equation}
Therefore estimate \eqref{26.4} yields:
\begin{equation}\label{26.6}
E\big(t,S^{(l+1)}\big)\leq\begin{cases}
C_9C_7^{l+1}G^{(l+1)}\exp\big(\omega_{l+1}(T-t)^{-\frac1{p+\mu}}\big),&\forall t\in(t_0^{(l)},T);\\
C_7^{l}G^{(l)}\exp\big(\omega_{l}(T-t_0^{(l)})^{-\frac1{p+\mu}}\big),& \forall
t\in(0,t_0^{(l)}], \end{cases}
\end{equation}
where $C_9=1+C_7^{-1}c_\zeta$. Estimate \eqref{26.6} is the starting point for the next
round of the new iterative estimating process. Let us introduce a new initial value
$t_0^{(l+1)}$ by the following analog of \eqref{25.5}:
\begin{equation}\label{26.7}
\omega_{l+1}\big(T-t_0^{(l+1)}\big)^{-\frac1{p+\mu}}=\Lambda_1.
\end{equation}
Then, repeating all computations which led from definition \eqref{25.5} to estimate
\eqref{26.4}, we obtain
\begin{equation}\label{26.8}
E\big(t,S^{(l+2)}\big)\leq\begin{cases}
C_9C_7^{l+2}G^{(l+2)}\exp\big(\omega_{l+2}(T-t)^{-\frac1{p+\mu}}\big)+E\big(t_0^{(l+1)},S^{(l+2)}\big),&\forall t\in(t_0^{(l+1)},T);\\
E\big(t_0^{(l+1)},S^{(l+2)}\big),& \forall t\in(0,t_0^{(l+1)}], \end{cases}
\end{equation}
The term $E\big(t_0^{(l+1)},S^{(l+2)}\big)$ we estimate from above by \eqref{26.6}:
\begin{equation}\label{26.88}
E\big(t_0^{(l+1)},S^{(l+2)}\big)\leq E\big(t_0^{(l+1)},S^{(l+1)}\big)\leq
C_9C_7^{l+1}G^{(l+1)}\exp\big(\omega_{l+1}(T-t_0^{(l+1)})^{-\frac1{p+\mu}}\big).
\end{equation}
Therefore using additionally estimates that are similar to \eqref{26.5} we derive from
\eqref{26.8} the following relation:
\begin{equation}\label{26.9}
E\big(t,S^{(l+2)}\big)\leq\begin{cases}
C_9^2C_7^{l+2}G^{(l+2)}\exp\big(\omega_{l+2}(T-t)^{-\frac1{p+\mu}}\big),&\forall t\in(t_0^{(l+1)},T);\\
C_9C_7^{l+1}G^{(l+1)}\exp\big(\omega_{l+1}(T-t_0^{(l+1)})^{-\frac1{p+\mu}}\big),&
\forall t\in(0,t_0^{(l+1)}]. \end{cases}
\end{equation}
Realizing $k$ such a rounds we arrive at
\begin{equation}\label{26.10}
E\big(t,S^{(l+k)}\big)\leq\begin{cases}
C_9^kC_7^{l+k}G^{(l+k)}\exp\big(\omega_{l+k}(T-t)^{-\frac1{p+\mu}}\big),&\forall t\in(t_0^{(l+k)},T);\\
C_9^{k-1}C_7^{l+k-1}G^{(l+k-1)}\exp\big(\omega_{l+k-1}(T-t_0^{(l+k-1)})^{-\frac1{p+\mu}}\big),&
\forall t\in(0,t_0^{(l+k-1)}], \end{cases}
\end{equation}
where $t_0^{(l+k)}$ are defined by
\begin{equation}\label{26.11}
\omega_{l+k}\big(T-t_0^{(l+k)}\big)^{-\frac1{p+\mu}}=\Lambda_1\ \ \ \forall k\in{N},\
\Lambda_1\text{ is from \eqref{25.5}}.
\end{equation}
Now we have to define the number $k$ of iterations. This number depends on $t$, namely
$k=k(t)$ is defined by relations:
\begin{equation}\label{26.12}
\begin{aligned}
&\omega\zeta^{l+k}\big(T-t\big)^{-\frac1{p+\mu}}\geq\Lambda_1, \
\omega\zeta^{l+k+1}\big(T-t\big)^{-\frac1{p+\mu}}<\Lambda_1\Rightarrow
\\&\omega\zeta^{l+k}\big(T-t\big)^{-\frac1{p+\mu}}=(1+\xi_0)\Lambda_1,\
\xi_0\leq\zeta^{-1}-1\Rightarrow
\\& l+k=\alpha_1+\alpha_2\ln(T-t)^{-1},\quad\alpha_1=l_1\ln\omega+l_2,\ l_1=\big(\ln\zeta^{-1}\big)^{-1},
\\&\,l_2=\frac{\ln\Lambda_1^{-1}+\ln(1+\xi_0)^{-1}}{\ln\zeta^{-1}},\quad\alpha_2=\big((p+\mu)\ln\zeta^{-1}\big)^{-1},
\end{aligned}
\end{equation}
where $\Lambda_1$ is from \eqref{25.5}. Then due to \eqref{26.12} estimate \eqref{26.10}
yields
\begin{equation}\label{26.13}
\begin{aligned}
&E\big(t,S^{(l+k(t))}\big)\leq
(C_9C_7)^{l+k}G^{(l+k)}\exp\big(\omega_{l+k}(T-t)^{-\frac1{p+\mu}}\big)\leq
\\&\leq(C_9C_7)^{\alpha_1+\alpha_2\ln(T-t)^{-1}}G^{(\infty)}(s_1)\exp\big(\Lambda_1(1+\xi_0)\big)\leq
\alpha_3(T-t)^{-\alpha_4}G^{\infty}(s_1),
\\&\alpha_3=l_3\omega^{l_4},\ l_3=\exp\big((1+\xi_0)\Lambda_1\big)(C_7C_9)^{l_2},\ l_4=l_1\ln(C_7C_9).
\end{aligned}
\end{equation}
Using \eqref{18.27}, \eqref{18.30} we derive from \eqref{26.13}:
\begin{equation}\label{26.14}
E\big(t,S^{(\infty)}\big):=E(t,s_1(1-\rho)^{-1})\leq
E\big(t,S^{(l+k(t))}\big)\leq\alpha_3(T-t)^{-\alpha_4}G^{(\infty)}(s_1),
\end{equation}
which yields
\begin{equation}\label{26.15}
E(t,s)\leq\alpha_3(T-t)^{-\alpha_4}G^{(\infty)}(s(1-\rho))=\alpha_3(T-t)^{-\alpha_4}D(s)\
\ \ \forall t<T,\ \forall\,s\in(0,(1-\rho)s'),
\end{equation}
where $D(s)=\exp\Big(C_8\omega^{\frac{p+\mu}\mu}s^{-\frac{p+1}\mu}\Big)$, $C_8$ is from
\eqref{18.29}.

\section{Proof of Theorem \ref{Th.1}: limiting blow-up profile of solution}

Thus we started with the global exponential blow-up estimate \eqref{global} of a
solution $u$ and obtained the improved power-like blow-up estimate \eqref{26.15} of a
solution in an arbitrary strongly interior subdomain $\Omega(s)$. Now starting with this
last estimate we deduce a sharp finite upper energy estimate of a solution $u$ in
$\Omega(s)$. Let us fix an arbitrary $\bar{s}\in(0,(1-\rho)s')$ and introduce a function
\begin{equation}\label{bb1}
   v(t,x):=D(\bar{s})^{-\frac1{p+1}}u(t,x)\ \ \
\forall(t,x)\in(0,T)\times\Omega(\bar{s}),
\end{equation}
where $u$ is a solution under consideration. Let us define
\begin{equation}\label{bb2}
\begin{aligned}
&\overline{E}(t,s):=\int_0^t\int_{\Omega(s)}|\nabla_xv(\tau,x)|^{p+1}dxd\tau=D(\bar{s})^{-1}E(t,s),
\\&\overline{h}(t,s):=\sup_{0<\tau<t}\int_{\Omega(s)}|v(\tau,x)|^{p+1}dx=D(\bar{s})^{-1}h(t,s)\ \ \ \forall s>\bar{s},\ \forall
t<T,
\end{aligned}
\end{equation}
where $E(t,s)$, $h(t,s)$ are energy functions connected with a solution $u$. Let
$\{\overline{E}_j(s)\}$, $\{\overline{h}_j(s)\}$ be families of energy functions
connected with functions from \eqref{bb2} and shifts $\{\Delta_j\}$. It is clear that
these functions satisfy the analogs of relations \eqref{28.7}--\eqref{28.9}, namely
\begin{equation}\label{bb3}
\overline{E}_j(s)\leq C_1\overline{h}_{j-1}(s)+C_2\Delta_j^\frac1{p+1}
    \biggr(-\frac d{ds}\overline{E}_j(s)\biggr),\quad j=1,2,... ,
\end{equation}
\begin{equation}\label{bb4}
\overline{h}_j(s)\leq(1+\gamma)\overline{h}_{j-1}(s)+
C_3\gamma^{-\frac1{p+1}}\Delta_j^\frac1{p+1}
    \biggr(-\frac d{ds}\overline{E}_j(s)\biggr),
\end{equation}
\begin{equation}\label{bb5}
\overline{h}_j(s)+\overline{E}_j(s)\leq
\overline{C}_1\overline{h}_{j-1}(s)+\overline{C}_2\Delta_j^\frac1{p+1} \biggr(-\frac
d{ds}\overline{E}_j(s)\biggr)\quad \forall\text{ a.a. }s\in(\bar{s},(1-\rho)s').
\end{equation}
Moreover due to \eqref{26.14} the following "initial" condition holds:
\begin{equation}\label{init1}
\overline{E}_j(\bar{s})\leq\overline{E}(t_j,\bar{s})\leq\alpha_3(T-t_j)^{-\alpha_4}\ \ \
\forall j\geq1,
\end{equation}
Without loss of the generality we suppose that $\alpha_4>(p+1)^{-1}$. Then introduce
positive numbers $\beta,\xi$:
\begin{equation}\label{11.1}
\beta>\alpha_4,\ \ 0<\xi<1:\ \theta_0:=(1+\gamma)\xi^{\alpha_4-\frac1{p+1}}<1.
\end{equation}
Due to the monotonicity of the function $\overline{E}(T,s)$ we can find
$\bar{\bar{s}}=\bar{\bar{s}}(\bar{s})>\bar{s}$ such that the following inequality holds:
\begin{equation}\label{11.22*}
\overline{E}(T,s)>2\alpha_3T^{-\alpha_4}\xi^{-\beta}\ \ \forall
s\in[\bar{s},\bar{\bar{s}}].
\end{equation}
Really if such a value $\bar{\bar{s}}>\bar{s}$ does not exist then
$E(T,\bar{s})\leq2\alpha_3T^{-\alpha_4}\xi^{-\beta}D(\bar{s})$ and the statement of
theorem \ref{Th.1} is valid for $s=\bar{s}$ with
$c_1=2\alpha_3T^{-\alpha_4}\xi^{-\beta}$ and $c_2=C_8$. Therefore we have to continue
the proof of theorem \ref{Th.1} only if for any $\varepsilon>0$ there exists
$\bar{s}=\bar{s}(\varepsilon)$ such that
\begin{equation*}
E(T,\bar{s})>2\alpha_3T^{-\alpha_4}\xi^{-\beta}D(\bar{s}),\ \
\bar{s}=\bar{s}(\varepsilon)\rightarrow0\text{ as }\varepsilon\rightarrow0.
\end{equation*}
It is clear that for such a value $\bar{s}$ there exists $\bar{\bar{s}}>\bar{s}$ that
satisfies \eqref{11.22*}. Now we introduce a family of continuous functions
\begin{equation}\label{11.23*}
\Gamma_{\tilde{s}}(\cdot):[0,t']\rightarrow[t_1,T]\ \ \
\forall\tilde{s}\in(\bar{s},\bar{\bar{s}}]
\end{equation}
by the following relation:
\begin{equation}\label{11.3}
\big(\Gamma_{\tilde{s}}(t)-t\big)^{-\beta}=\frac{\xi^\beta}{\alpha_3T^{\beta-\alpha_4}}
\Big(\overline{E}(\Gamma_{\tilde{s}}(t),\tilde{s})-\overline{E}(t,\tilde{s})\Big).
\end{equation}
Here $t_1=t_1(\tilde{s})=\Gamma_{\tilde{s}}(0)$ is defined by the equality:
\begin{equation}\label{11.4}
t_1^{-\beta}=\frac{\xi^\beta}{\alpha_3T^{\beta-\alpha_4}}\overline{E}(t_1,\tilde{s}),
\end{equation}
and $t'$ is defined by the relation:
\begin{equation}\label{11.5}
(T-t')^{-\beta}=\frac{\xi^\beta}{\alpha_3T^{\beta-\alpha_4}}
\Big(\overline{E}(T,\tilde{s})-\overline{E}(t',\tilde{s})\Big).
\end{equation}
Due to definition \eqref{11.4}, condition \eqref{11.1} and property \eqref{11.22*} we
have
\begin{equation}\label{11.61}
\overline{E}(t_1,\tilde{s})t_1^{\beta}=\frac{\alpha_3T^{\beta-\alpha_4}}{\xi^\beta}<
2^{-1}\overline{E}(T,\tilde{s})T^{\beta}\ \ \
\forall\tilde{s}\in(\bar{s},\bar{\bar{s}}].
\end{equation}
Therefore due to the strong monotonicity of a function
$\varphi(t):=\overline{E}(t,\tilde{s})t^{\beta}$ it follows that $t_1(\tilde{s})<T$
$\forall\tilde{s}\in(\bar{s},\bar{\bar{s}}]$. Remark that definition \eqref{11.5} yields
\begin{equation}\label{11.6}
t'<T\text{ if }\sup_{t\rightarrow T}\overline{E}(t,\tilde{s})<\infty,
\end{equation}
\begin{equation}\label{11.7}
t'=T\text{ if }\sup_{t\rightarrow T}\overline{E}(t,\tilde{s})=\infty.
\end{equation}
Now we can conclude that the function $\Gamma_{\tilde{s}}(\cdot)$ determines a strongly
monotonically increasing sequence $\{t_j\}$ by the following relation:
\begin{equation}\label{11.9}
t_j:=\Gamma_{\tilde{s}}(t_{j-1}),\ \ \ j=1,2,...,\ t_0=0.
\end{equation}
Moreover this sequence is infinite and $t_j\rightarrow T$ as $j\rightarrow\infty$ in the
case \eqref{11.7}. In the case \eqref{11.6} this sequence is finite and there exists a
number $j_0$ such that
\begin{equation}\label{11.8}
t_{j_0}=\Gamma_{\tilde{s}}(t_{j_0-1})>t',\ \ \ t_{j_0-1}\leq t'.
\end{equation}
Now we introduce new shifts $\{\Delta_j\}=\{\Delta_j(\tilde{s})\}$ for the system
\eqref{28.7}, \eqref{28.8}, namely
\begin{equation}\label{11.10}
\Delta_j=\Delta_j(\tilde{s})=t_j-t_{j-1}, \ \ \ \{t_j\}\text{ is from \eqref{11.9}}.
\end{equation}
Due to definition \eqref{11.3} of the function $\Gamma_{\tilde{s}}(t)$ and the estimate
\eqref{26.14} (and consequently \eqref{26.15}) the following inequalities hold:
\begin{equation*}
\Delta_j^{-\beta}=\frac{\xi^{\beta}}{\alpha_3T^{\beta-\alpha_4}}\big(\overline{E}(t_j,\tilde{s})-\overline{E}(t_{j-1},\tilde{s}))\big)\leq
\frac{\xi^{\beta}\alpha_3(T-t_j)^{-\alpha_4}}{\alpha_3T^{\beta-\alpha_4}}
\leq\xi^{\beta}T^{\alpha_4-\beta}(T-t_j)^{-\alpha_4}\ \ \ \forall j\in{N}.
\end{equation*}
Therefore,
\begin{equation}\label{11.11}
\begin{aligned}
&\Delta_j=\Delta_j(\tilde{s})\geq\xi^{-1}T^{1-\frac{\alpha_4}\beta}(T-t_j)^{\frac{\alpha_4}\beta}=
\xi^{-1}T^{1-\frac{\alpha_4}\beta}\biggr(\frac{(T-t_j)}{T}\biggr)^{\frac{\alpha_4}\beta}T^{\frac{\alpha_4}\beta}\geq
\\&\geq\xi^{-1}T\biggr(\frac{(T-t_j)}{T}\biggr)\geq \xi^{-1}\Delta_{j+1}\ \ \ \forall
j\in{N},\,\forall\,\tilde{s}\in(\bar{s},\bar{\bar{s}}).
\end{aligned}
\end{equation}
Let now $\overline{E}_j(s)$, $\overline{h}_j(s)$ be energy functions of our solution
$u$, corresponding to the sequence of shifts \eqref{11.10}. These functions satisfy
system \eqref{28.7}, \eqref{28.8} for almost all $s\in(\tilde{s},s')$. Now introduce new
energy functions $\overline{A}_j(s)$, $\overline{H}_j(s)$ by the relations
\begin{equation}\label{11.12}
\overline{A}_j(s):=\Delta_j(\tilde{s})^{\alpha_4}\overline{E}_j(s),\
\overline{H}_j(s):=\Delta_j(\tilde{s})^{\alpha_4}\overline{h}_j(s)\ \ \ \forall
s>\tilde{s},\ j\in{N}.
\end{equation}
It is easy to see that these functions satisfy the following inequalities for almost all
$s\in(\tilde{s},s')$:
\begin{equation}\label{11.13}
\begin{aligned}
&\overline{A}_j(s)\leq
C_1\biggr(\frac{\Delta_j}{\Delta_{j-1}}\biggr)^{\alpha_4}\overline{H}_{j-1}(s)+C_2\Delta_j^\frac1{p+1}
\biggr(-\frac d{ds}\overline{A}_j(s)\biggr)\ \ \ \forall j\in{N},
\\&\overline{H}_i(s)\leq
\lambda_i\overline{H}_{i-1}(s)+C_3\gamma^{-\frac1{p+1}}\Delta_i^\frac1{p+1}
\biggr(-\frac d{ds}\overline{A}_i(s)\biggr)\ \ \ \forall i\in{N},
\end{aligned}
\end{equation}
where $\lambda_i=(1+\gamma)\Big(\frac{\Delta_i}{\Delta_{i-1}}\Big)^{\alpha_4}$. It is
easy to check that
\begin{equation}\label{11.14}
\lambda_j\lambda_{j-1}...\lambda_{i+1}\Delta_i^\frac1{p+1}=
(1+\gamma)^{j-i}\Delta_j^\frac1{p+1}
\biggr(\frac{\Delta_j}{\Delta_{i}}\biggr)^{\alpha_4-\frac1{p+1}}.
\end{equation}
Using relation \eqref{11.14} and iterating system \eqref{11.13} we get the following
inequalities:
\begin{equation}\label{11.13*}
\begin{aligned}
&\overline{A}_j(s)\leq
C_1(1+\gamma)^{j-1}\biggr(\frac{\Delta_j}{\Delta_0}\biggr)^{\alpha_4}\overline{H}_0(s)+
\\&+\overline{C}_3\gamma^{-\frac1{p+1}}\Delta_j^\frac1{p+1}
\biggr(\sum_{i=1}^j(1+\gamma)^{j-i}\biggr(\frac{\Delta_j}{\Delta_i}\biggr)^{\alpha_4-\frac1{p+1}}
\big(-\overline{A}'_i(s)\big)\biggr)\ \ \ \forall s\in(\tilde{s},s'),
\end{aligned}
\end{equation}
where $\overline{C}_3=\max\{C_2\gamma^\frac1{p+1},(1+\gamma)^{-1}C_1C_3\}$,
$\overline{H}_0(s)=\Delta_0^{\alpha_4}\overline{h}_0(s)$, $\Delta_0=\xi^{-1}\Delta_1$.
Now introduce one more family of energy functions:
\begin{equation}\label{11.17}
U_j(s):=
\sum_{i=1}^j(1+\gamma)^{j-i}\biggr(\frac{\Delta_j}{\Delta_i}\biggr)^{\alpha_4-\frac1{p+1}}\overline{A}_i(s),
\ \ \ j=1,2,...,\ \forall s>\tilde{s}.
\end{equation}
As is easily verified that these functions satisfy relations:
\begin{equation}\label{11.18}
U_j(s)-\overline{A}_j(s)=\theta^{(j)}U_{j-1}(s),\ \ \ j=1,2,...,
\end{equation}
where due to \eqref{11.11}
$\theta^{(j)}:=(1+\gamma)\big(\frac{\Delta_j}{\Delta_{j-1}}\big)^{\alpha_4-\frac1{p+1}}
\leq(1+\gamma)\xi^{\alpha_4-\frac1{p+1}}=\theta_0$ with $\theta_0$ from \eqref{11.1}.
Using \eqref{11.1} we derive from \eqref{11.18}:
\begin{equation}\label{11.19}
U_j(s)\leq
C_1T^{\alpha_4}\bar{\lambda}^{j-1}\overline{h}_0(s)+\theta_0U_{j-1}(s)+\overline{C}_3\gamma^{-\frac1{p+1}}
\Delta_j^{\frac1{p+1}}(-U'_j(s))\quad\forall\text{ a.a. }s\in(\tilde{s},s'),
\end{equation}
where $\bar{\lambda}=(1+\gamma)\xi^{\alpha_4}<1$. Let us estimate from above the
"initial" value $U_j(\tilde{s})$ of the function $U_j(s)$. Due to definitions
\eqref{11.10} and \eqref{11.3} we get:
\begin{equation}\label{11.20}
\begin{aligned}
&U_j(\tilde{s})=\sum_{i=1}^j(1+\gamma)^{j-i}\biggr(\frac{\Delta_j}{\Delta_i}\biggr)^{\alpha_4-\frac1{p+1}}
{\Delta_i}^{\alpha_4}\overline{E}_i(\tilde{s})=
\\&=\alpha_3T^{\beta-\alpha_4}\xi^{-\beta}
\sum_{i=1}^j(1+\gamma)^{j-i}\biggr(\frac{\Delta_j}{\Delta_i}\biggr)^{\alpha_4-\frac1{p+1}}
{\Delta_i}^{-\beta}{\Delta_i}^{\alpha_4}=
\\&=\alpha_3T^{\beta-\alpha_4}\xi^{-\beta}{\Delta_j}^{-(\beta-\alpha_4)}
\sum_{i=1}^j(1+\gamma)^{j-i}\biggr(\frac{\Delta_j}{\Delta_i}\biggr)^{\beta-\frac1{p+1}}
\leq\alpha_3(1-\theta_0)^{-1}T^{\beta-\alpha_4}\xi^{-\beta}{\Delta_j}^{-(\beta-\alpha_4)}.
\end{aligned}
\end{equation}
It is easy to verify that functions
$\overline{U}_j(s):=U_j(s)-(1-\theta_0)^{-1}C_1T^{\alpha_4}\overline{h}_0(\tilde{s})$
satisfy the homogeneous ODI \eqref{11.19} for almost all $s\in(\tilde{s},s')$:
\begin{equation}\label{11.21}
\overline{U}_j(s)\leq \theta_0\overline{U}_{j-1}(s)+\overline{C}_3
\Delta_j^{\frac1{p+1}}(-\overline{U}'_j(s)).
\end{equation}
Moreover these functions $\overline{U}_j(s)$ satisfy the "initial" condition
\eqref{11.20}:
\begin{equation*}
\overline{U}_j(\tilde{s})\leq\alpha_3(1-\theta_0)^{-1}T^{\beta-\alpha_4}
\xi^{-\beta}{\Delta_j}^{-(\beta-\alpha_4)}.
\end{equation*}
Therefore in virtue of lemmas 9.2.7 --- 9.2.9 from \cite{KSSh} functions
$\overline{U}_j(s)$ satisfy the following uniform with respect to $j\in{N}$ estimate:
\begin{equation}\label{11.22}
\overline{U}_j(s)\leq\alpha_3C_{10}\overline{U}(s-\tilde{s})\ \ \ \forall
s\in(\tilde{s},s'),\ j=1,2,...,
\end{equation}
where $C_{10}=(1-\theta_0)^{-1}T^{\beta-\alpha_4}\xi^{-\beta}
\big(\overline{C}_3(\beta-\alpha_4)(p+1)e^{-1}(1-\theta_0)^{-1}\big)^{(\beta-\alpha_4)(p+1)}$,
$\overline{U}(s):=s^{-(\beta-\alpha_4)(p+1)}$ and consequently:
\begin{equation}\label{11.23}
U_j(s)\leq C_{11}+\alpha_3C_{10}\overline{U}(s-\tilde{s})\ \ \ \forall
s\in(\tilde{s},s'),
\end{equation}
where $C_{11}=(1-\theta_0)^{-1}C_1T^{\alpha_4}h_0(\tilde{s})$. Let us define a value
$s^{(1)}>0$ by the relation:
\begin{equation}\label{11.24}
\overline{U}(s)\geq\alpha_3^{-1}C_{11}C_{10}^{-1}\ \ \ \forall s\leq
s^{(1)}\quad\Rightarrow\quad
s^{(1)}=\left(\alpha_3C_{10}C_{11}^{-1}\right)^\frac1{(\beta-\alpha_4)(p+1)}.
\end{equation}
Then it follows from \eqref{11.23} that:
\begin{equation}\label{11.25}
\overline{A}_j(s)\leq U_j(s)\leq2\alpha_3C_{10}\overline{U}(s-\tilde{s})\ \ \ \forall
s\in(\tilde{s},s''),\,s''=\min(s',s^{(1)}+\tilde{s}),\,j=1,2,...,
\end{equation}
which yields in virtue of \eqref{11.12} and \eqref{11.3}:
\begin{equation}\label{11.26}
\overline{E}_j(s)\leq 2\alpha_3C_{10}\overline{U}(s-\tilde{s})
\big(\Delta_j(\tilde{s})\big)^{-\alpha_4}\ \ \ \forall
\tilde{s}\in(\bar{s},\bar{\bar{s}}], \ \forall s\in(\tilde{s},s''),
\end{equation}
Now we sum inequalities \eqref{11.26} from $j=1$ up to $j=i$. Using the property
\eqref{11.11} we get:
\begin{equation*}
\overline{E}(t_i,s)\leq2\alpha_3C_{10}\overline{U}(s-\tilde{s})
\sum_{j=1}^i\Delta_j^{-\alpha_4}\leq
2\alpha_3C_{10}\overline{U}(s-\tilde{s})\Delta_i^{-\alpha_4}(1-\xi^{\alpha_4})^{-1}.
\end{equation*}
Using additionally definitions \eqref{11.3}, \eqref{11.10} we derive from the last
inequality:
\begin{equation*}
\overline{E}(t_i,s)\leq C_{12}\alpha_3^\frac{\beta-\alpha_4}\beta
\overline{U}(s-\tilde{s}) \overline{E}_i(\tilde{s})^\frac{\alpha_4}\beta,\ \
C_{12}=2(1-\xi^{\alpha_4})^{-1}C_{10}\xi^{\alpha_4}T^{-\frac{(\beta-\alpha_4)\alpha_4}\beta}\alpha_3^{-\frac{\alpha_4}{\beta}},
\end{equation*}
and as consequence,
\begin{equation}\label{11.27}
\overline{E}(t_i,s)\leq C_{12}\alpha_3^\frac{\beta-\alpha_4}\beta
\overline{U}(s-\tilde{s}) \overline{E}(t_i,\tilde{s})^\frac{\alpha_4}\beta\ \ \ \forall
\tilde{s}\in(\bar{s},s^{(3)}), \ \forall s\in(\tilde{s},s'''),\ i=1,2,...\ ,
\end{equation}
where $s'''=\min(s',s^{(1)})=\min(s^{(1)},s_\Omega,s_\omega)\leq s''$,
$s^{(3)}=\min(\bar{\bar{s}},s''')=\min(\bar{\bar{s}},s^{(1)},s_\Omega,s_\omega)$. Now we
have to establish the relation of the type \eqref{11.27} not only for
$t=t_i(\tilde{s})$, $i=1,2,...$, but also for an arbitrary $t<T$. Firstly we consider a
value $t_1$ defined by the equality \eqref{11.4} as a function $t_1=t_1(\tilde{s})$. Due
to the monotonic increasing of the energy function $E(t,s)$ with respect to $t$ and its
monotonic decreasing with respect to $s$, definition \eqref{11.4} guarantees that
$t_1(\tilde{s})$ is the monotonically increasing function. Moreover, it follows from
\eqref{11.61} that
\begin{equation}\label{15.1}
t_1(\tilde{s})<T\ \ \ \forall\tilde{s}\in(\bar{s},s^{(3)}),
\end{equation}
where $s^{(3)}$ is from \eqref{11.27}. For an arbitrary $\tilde{s}$ from \eqref{15.1}
let $\{t_i\}=\{t_i(\tilde{s})\}$ be the sequence defined by \eqref{11.9}. It is easy to
see that the mapping $\Gamma_{\tilde{s}}(\cdot)$ defined by \eqref{11.3} maps the
segment $[t_{i-1}(\tilde{s}),t_i(\tilde{s})]$ into
$[t_{i}(\tilde{s}),t_{i+1}(\tilde{s})]$ continuously, monotonically and bijectively for
any $i\geq1$ in the case \eqref{11.7} and for any $i:1\leq i\leq j_0-1$
 in the case \eqref{11.6}. Moreover in the case \eqref{11.6} it also maps bijectively
 segment $[t_{j_0-1},\Gamma_{\tilde{s}}^{-1}(t')]$ into $[t_{j_0},T]$. Let now $t$ be an
 arbitrary point from the interval $(t_1(\tilde{s}),T)$.
For the definiteness we can assume that
 $t\in(t_{k-1},t_k)$ with some $k\in{N}$ in the case \eqref{11.7} or  $t\in(t_{k-1},t_k)$ with some $k\leq
 j_0$ or $t\in(t_{j_0},T)$ in the case \eqref{11.6}. Due to bijectivity of the map
 $\Gamma_{\tilde{s}}(\cdot)$ we can reconstruct the finite sequence
 $\{\bar{t}_i(\tilde{s})\}$ as follows:
 \begin{equation}\label{15.3}
\bar{t}_i:=\Gamma_{\tilde{s}}^{-1}(\bar{t}_{i+1}),\ \ i=k,k-1,\ldots,0, \text{ where
}\bar{t}_{k+1}:=t,\ \bar{t}_{i}\in(t_i,t_{i+1}),\ \bar{t}_0\in(0,t_1(\tilde{s})).
 \end{equation}
Hence we obtain the following sequence of shifts $\bar{\Delta}_j(\tilde{s})$:
\begin{equation*}
\bar{\Delta}_j^{-\beta}:=(\bar{t}_{j}-\bar{t}_{j-1})^{-\beta}=\frac{\xi^\beta
T^{\alpha_4-\beta}}{\alpha_3}
\Big(\overline{E}(\bar{t}_{j},\tilde{s})-\overline{E}(\bar{t}_{j-1},\tilde{s})\Big).
\end{equation*}
As above in \eqref{11.11} we show that these shifts satisfy the inequality
$\bar{\Delta}_{j+1}<\xi\bar{\Delta}_j$ $\forall j\leq k$. Using these shifts we
introduce new energy functions $\overline{E}_j^{(1)}(s)$, $\overline{h}_j^{(1)}(s)$.
Using these energy functions and repeating all computations which led from the estimate
\eqref{26.15} to the relation \eqref{11.27}, we arrive at:
\begin{equation*}
\overline{E}(\bar{t}_i,s)-\overline{E}(\bar{t}_0,s)\leq
C_{12}\alpha_3^\frac{\beta-\alpha_4}\beta \overline{U}(s-\tilde{s})
\overline{E}(\bar{t}_i,\tilde{s})^\frac{\alpha_4}\beta,\ \ \ \forall i\leq k+1,
\end{equation*}
and as consequence for $i=k+1$:
\begin{equation}\label{15.4}
\overline{E}(t,s)-\overline{E}(\bar{t}_0,s)\leq
C_{12}\alpha_3^\frac{\beta-\alpha_4}\beta \overline{U}(s-\tilde{s})
\overline{E}(t,\tilde{s})^\frac{\alpha_4}\beta\ \ \
\forall\,\tilde{s}\in(\bar{s},s^{(3)}),\,\forall s\in(\tilde{s},s''').
\end{equation}
Remark that a point $t$ in the last inequality is arbitrary: $t<T$. Since
$\bar{t}_0(\tilde{s})<t_1(\tilde{s})$ then summing inequality \eqref{11.27} by $i=1$ and
inequality \eqref{15.4}  we obtain the following relation:
\begin{equation}\label{15.5}
\begin{aligned}
&\overline{E}(t,s)\leq
\overline{E}(t,s)-\overline{E}(\bar{t}_0(\tilde{s}),s)+\overline{E}(t_1(\tilde{s}),s)\leq
C_{12}\alpha_3^\frac{\beta-\alpha_4}\beta\overline{U}(s-\tilde{s})\times
\\&\times\big(\overline{E}(t,\tilde{s})^\frac{\alpha_4}\beta+\overline{E}(t_1(\tilde{s}),\tilde{s})^\frac{\alpha_4}\beta\big)
<2C_{12}\alpha_3^\frac{\beta-\alpha_4}\beta \overline{U}(s-\tilde{s})
\overline{E}_(t,\tilde{s})^\frac{\alpha_4}\beta,
\end{aligned}
\end{equation}
which is true for an arbitrary $t<T$ and arbitrary $s,\tilde{s}:\bar{s}<\tilde{s}<s\leq
s^{(3)}$. Let us introduce new variables: $v=s-\bar{s}$, $w=\tilde{s}-\bar{s}>0$ and
shifted energy function $\overline{E}_{\bar{s}}(t,v):=\overline{E}(t,\bar{s}+v)$. Then
the relation \eqref{15.5} can be written as follows:
\begin{equation}\label{Stamp1}
\overline{E}_{\bar{s}}(t,v)\leq
c\overline{U}(v-w)\overline{E}_{\bar{s}}(t,w)^\frac{\alpha_4}\beta\quad
\forall\,0<w<v<s^{(3)}-\bar{s},\,\forall\,t<T,
\end{equation}
with $c=2C_{12}\alpha_3^\frac{\beta-\alpha_4}\beta$. Due to Stampacchia lemma
(\cite{Stamp}, see also lemma 9.3.2 in \cite{KSSh}) and the structure \eqref{11.22} of
the function $\overline{U}(\cdot)$ the relation \eqref{Stamp1} yields the following
uniform with respect to $t<T$ estimate
$\overline{E}_{\bar{s}}(t,v)\leq\alpha_3C_{13}v^{-\beta(p+1)}$ $\forall\,v\leq
s^{(3)}-\bar{s}$, $\forall\,t<T$, which yields:
\begin{equation}\label{15.6}
\overline{E}(t,s)\leq\alpha_3C_{13}(s-\bar{s})^{-\beta(p+1)}\ \ \ \forall\,t<T,\,\forall
s:\bar{s}<s<s^{(3)},
\end{equation}
where $C_{13}:=2^{(p+1)\beta^2}C_{12}^\frac\beta{\beta-\alpha_4}$. Inserting expressions
\eqref{bb2} for $\overline{E}(t,s)$ and \eqref{26.15} for $D(s)$, we derive from
\eqref{15.6}:
\begin{equation}\label{15.7}
\begin{aligned}
&E(t,s)\leq\alpha_3 C_{13}f(\bar{s},s):=
\alpha_3C_{13}\exp\big(C_8\omega^\frac{p+\mu}\mu\bar{s}^{-\frac{p+1}\mu}\big)(s-\bar{s})^{-\beta(p+1)}\quad\forall
s\in(\bar{s},s^{(3)}),
\\&\forall\,\bar{s}<s^{(3)},\,\forall\,t\leq T.
\end{aligned}
\end{equation}
Let us rewrite estimate \eqref{15.7} in the form:
\begin{equation}\label{ar1}
\begin{aligned}
&E(t,s)\leq\alpha_3C_{13}\omega^{-\beta(p+\mu)}\varphi(\bar{r},r),\ \
r:=s\omega^{-\frac{p+\mu}{p+1}},
\\&\bar{r}:=\bar{s}\omega^{-\frac{p+\mu}{p+1}},\ \
\varphi(\bar{r},r):=\exp\big(C_8\bar{r}^{-\frac{p+1}\mu}\big)(r-\bar{r})^{-\beta(p+1)}.
\end{aligned}
\end{equation}
As is easy to see for an admissible value $s$ in the estimate \eqref{15.7} the following
restriction holds: $s\leq s_\omega=b_1\omega^{\frac{p+\mu}{p+1}}$, where
$b_1=\left(\frac{B_1}{B_2}\right)^\frac\mu{p+1}$. Therefore $s\leq
b_1\omega^{\frac{p+\mu}{p+1}}$ and, consequently $r=s\omega^{-\frac{p+\mu}{p+1}}\leq
b_1$. Let us fix an arbitrary point $s<s^{(3)}$. Then only two cases are possible with
respect to a point $\bar{s}=2^{-1}s$, namely
\\1) $E(T,\bar{s})\leq2\alpha_3T^{-\alpha_4}\xi^{-\beta}D(\bar{s})$, 2)
$E(T,\bar{s})>2\alpha_3T^{-\alpha_4}\xi^{-\beta}D(\bar{s})$.
\\In the case 1) we have the following estimate $E(T,s)\leq
E(T,\bar{s})\leq2\alpha_3T^{-\alpha_4}\xi^{-\beta} \exp\big(2^\frac{p+1}\mu
C_8\omega^{\frac{p+\mu}\mu}s^{-\frac{p+1}\mu}\big)$, which corresponds to the desired
estimate \eqref{main1} in a point $s$. And we have also two possibilities in the case
2), namely
\\a) there is $\bar{\bar{s}}=\bar{\bar{s}}(\bar{s})<s$ such that
$E(T,\bar{\bar{s}})=2\alpha_3T^{-\alpha_4}\xi^{-\beta}D(\bar{s})$ and
$E(T,s')>2\alpha_3T^{-\alpha_4}\xi^{-\beta}D(\bar{s})$
$\forall\,s'\in[\bar{s},\bar{\bar{s}})$,
\\b) $E(T,s')>2\alpha_3T^{-\alpha_4}\xi^{-\beta}D(\bar{s})\quad\forall\,s'\leq s$.
\\In the case a) we have the following estimate:
\begin{equation}\label{31.10}
E(T,s)\leq E(T,\bar{\bar{s}})=2\alpha_3T^{-\alpha_4}\xi^{-\beta}D(\bar{s})=
2\alpha_3T^{-\alpha_4}\xi^{-\beta} \exp\big(2^\frac{p+1}\mu
C_8\omega^{\frac{p+\mu}\mu}s^{-\frac{p+1}\mu}\big),
\end{equation}
which corresponds to the estimate in the case 1). Finally in the case b) the estimate
\eqref{15.7} yields:
\begin{equation}\label{ar5}
\begin{aligned}
&E(t,s)\leq \alpha_3C_{13}\omega^{-\beta(p+1)}
\exp\big(C_8\bar{r}^{-\frac{p+1}\mu}\big)(r-\bar{r})^{-\beta(p+1)}=
\\&=\alpha_3C_{13}\omega^{-\beta(p+1)}
\exp\big(2^\frac{p+1}\mu C_8r^{-\frac{p+1}\mu}\big)2^{\beta(p+1)}r^{-\beta(p+1)}\leq
\\&\leq\alpha_3\omega^{-\beta(p+1)}2^{\beta(p+1)}C_{13}\tilde{k}
\exp\big(2^\frac{p+\mu+1}\mu C_8r^{-\frac{p+1}\mu}\big),
\end{aligned}
\end{equation}
where $\tilde{k}=\max_{r\leq b_1}\frac{r^{-\beta(p+1)}} {\exp\big(2^\frac{p+\mu+1}\mu
C_8r^{-\frac{p+1}\mu}\big)}$. Finally combining estimates \eqref{31.10}, \eqref{ar5} we
get
\begin{equation}\label{ar6}
\begin{aligned}
&E(t,s)\leq E_0(t,s):=C_{15}\exp\big(2^\frac{p+\mu+1}\mu
C_8\omega^{\frac{p+\mu}\mu}s^{-\frac{p+1}\mu}\big)\ \ \ \forall t\leq T,\ \forall
s:0<s<\min(s_\Omega,s_\omega,s^{(1)})
\\&C_{15}=\alpha_3\big(2T^{-\alpha_4}\xi^{-\beta}+2^{\beta(p+1)}C_{13}\tilde{k}\omega^{-\beta(p+1)}\big).
\end{aligned}
\end{equation}
In view of the above values of the constants $s_\Omega,\,s_\omega,\,s^{(1)}$ the last
estimate yields the desired estimate \eqref{main1} for the energy function $E(t,s)$ of a
solution $u$. The required estimate of $h(t,s)$ follows from the following simple
computation. Let us fix $0<\lambda=const<1,\,s>0$ and let $\xi_{\lambda,s}(r)$ be a
Lipschitz cut-off function:
\begin{equation*}
\begin{aligned}
&\xi_{\lambda,s}(r)=0\quad\text{if }r\leq\lambda s,\quad
\xi_{\lambda,s}(r)=1\quad\text{if }r\geq s,
\\&\xi_{\lambda,s}(r)=(r-\lambda s)(s(1-\lambda))^{-1}\quad\text{if }\lambda s<r<s.
\end{aligned}
\end{equation*}
Then inserting the function $\eta(t,x)=u(t,x)\xi_{\lambda,s}^\frac{p+1}p(d(x))$ for the
integral identity as a test function, we obtain after standard calculations the
following inequality:
\begin{equation*}
\begin{aligned}
&\frac p{p+1} \int_{\Omega(\lambda s)}|u(t,x)|^{p+1}\xi_{\lambda,s}^\frac{p+1}pdx+
d_0\int_0^t\int_{\Omega(\lambda s)}
|\nabla_xu(\tau,x)|^{p+1}\xi_{\lambda,s}^\frac{p+1}pdxd\tau\leq
\\&\leq d_1\left(\int_0^t\int_{\Omega(\lambda s)\setminus\Omega(s)}
|\nabla_xu|^{p+1}dxd\tau\right)^\frac p{p+1} \max_{0\leq\tau\leq t}
\left(\int_{\Omega(\lambda s)\setminus\Omega(s)}
|u(\tau,x)|^{p+1}\xi_{\lambda,s}^\frac{p+1}pdx\right)^\frac1{p+1}\times
\\&\times t^\frac1{p+1}
\max_{\Omega(\lambda s)\setminus\Omega(s)}|\nabla\xi_{\lambda,s}(d(x))| +\frac
p{p+1}\int_{\Omega(\lambda s)}|u(0,x)|^{p+1}\xi_{\lambda,s}^\frac{p+1}pdx,
\end{aligned}
\end{equation*}
which leads due to Young inequality to the estimate:
\begin{equation*}
\begin{aligned}
&\left(\frac p{p+1}-\varepsilon\right)\max_{0\leq\tau\leq t} \int_{\Omega(\lambda
s)}|u(\tau,x)|^{p+1}\xi_{\lambda,s}^\frac{p+1}pdx\leq \frac p{p+1}h_0(\lambda s)+
\\&+C(\varepsilon)\left(\frac{d_1T^\frac1{p+1}}{s(1-\lambda)}\right)^\frac{p+1}p
\int_0^t\int_{\Omega(\lambda s)\setminus\Omega(s)} |\nabla_xu|^{p+1}dxd\tau\quad
\forall\,t<T,\,\forall\,\varepsilon>0.
\end{aligned}
\end{equation*}
Due to \eqref{ar6} this estimate with $\varepsilon=\frac p{2(p+1)}$ yields:
\begin{equation*}
\frac p{2(p+1)}h(t,s)\leq\frac p{p+1}h_0(\lambda s)+ C\left(\frac
p{2(p+1)}\right)\left(\frac{d_1T^\frac1{p+1}}{s(1-\lambda)}\right)^\frac{p+1}p
E_0(t,\lambda s).
\end{equation*}
Optimizing this last estimate with respect to a free parameter $\lambda<1$ we derive the
desired estimate for $h(t,s).$

\section{Propagation of singularities of large solutions: proof of Theorem \ref{Th.2}}

\begin{definition}\label{def2}
A function $u(t,x)\in C_{loc}((0,T);L^{p+1}(\Omega))$ is called a weak (energy) solution
of equation~\eqref{da1} if:
\begin{equation*}
\begin{aligned}
&u(t,x)\in L^{p+1}_{loc}\left((0,T);W^{1,p+1}_{loc}(\Omega)\right)\cap
L^{\lambda+1}_{loc}\left((0,T)\times\Omega\right),
\\&(|u|^{p-1}u)_t\in L^{\frac{p+1}p}_{loc}\left((0,T);(W^{1,p+1}_{c}(\Omega))^*\right)+
L^{\frac{\lambda+1}\lambda}_{loc}\left((0,T);(L^{\lambda+1}_{c}(\Omega))^*\right)
\end{aligned}
\end{equation*}
and the following integral identity holds:
\begin{equation}\label{4.1}
\int_a^b\langle(|u|^{p-1}u)_t,\eta\rangle dt+
\int_a^b\int_{\Omega}\biggr[\sum_{i=1}^na_i(... ,\nabla
u)\eta_{x_i}+b(t,x)|u|^{\lambda-1}u\eta\biggr]dxdt=0
\end{equation}
for arbitrary $0<a<b<T$ and an arbitrary
\begin{equation*}
\eta(t,x)\in L^{p+1}_{loc}\left((0,T);W^{1,p+1}_{c}(\Omega)\right)\cap
L^{\lambda+1}_{loc}\left((0,T);L^{\lambda+1}_{c}(\Omega)\right),
\end{equation*}
where $W^{1,p+1}_{c}(\Omega)$, $L^{ \lambda+1}_{c}(\Omega)$\ are subspaces of
$W^{1,p+1}(\Omega)$, $L^{\lambda+1}(\Omega)$ \ of functions with the compact support in
$\Omega$, and $<,>$ is the pairing of elements from $W^{1,p+1}_{c}(\Omega)\cap L^{
\lambda+1}_{c}(\Omega)$\ and $\left(W^{1,p+1}_{c}(\Omega)\cap L^{
\lambda+1}_{c}(\Omega)\right)^*$.
\end{definition}
Let $\Omega(s)$ be a family of subdomains from \eqref{fam1}. Let us introduce an
additional family of cylindrical subdomains of $Q$:
\begin{equation}\label{4.3}
Q_\tau(s):=(s^q,\tau)\times\Omega(s)\ \ \ \forall s\in(0,s_\Omega),\ \forall \tau<T,\
1<q=const<1+\frac{p(\lambda+1)}{p+1}.
\end{equation}
Now we define the following energy functions connected with a solution $u$ of equation
\eqref{da1} under consideration:
\begin{equation}\label{4.4}
h^{(u)}_\tau(s)=h_\tau(s):=\int_{\Omega(s)}|u(\tau,x)|^{p+1}dx,\quad 0<\tau<T,
\,0<s<s_\Omega,
\end{equation}
\begin{equation}\label{4.5}
E^{(u)}_\tau(s)=E_\tau(s):=\int^\tau_{s^q}\int_{\Omega(s)}(|\nabla_x
u(t,x)|^{p+1}+a_1(t)g_1(d(x))|u|^{\lambda+1})dxdt.
\end{equation}
\begin{lemma}\label{Lem2.1}
Let $u$ be a solution of equation \eqref{da1}. Then energy functions \eqref{4.4},
\eqref{4.5} satisfy the following relation:
\begin{equation}\label{4.55}
\begin{aligned}
&B_\tau(s):=h_\tau(s)+E_\tau(s)\leq
C_1g_1(s)^{-\frac{1}{\lambda+1}}\Phi(\tau)^\frac{\lambda-p}{(\lambda+1)(p+1)}
\biggr(-\frac{d}{ds}E_\tau(s)\biggr)^\frac{1+p(\lambda+2)}{(\lambda+1)(p+1)}+
\\&+C_2g_1(s)^{-\frac{p+1}{\lambda+1}}
\biggr(-\frac{d}{ds}E_\tau(s)\biggr)^\frac{p+1}{\lambda+1}s^{-\frac{(q-1)(p+1)}{\lambda+1}}\
\ \ \text{for a.a. }s\in(0,s_\Omega),\,\text{where }\Phi(\tau)=\int_0^\tau
a_1(t)^{-\frac{p+1}{\lambda+1}}dt.
\end{aligned}
\end{equation}
\end{lemma}
\begin{proof}
Let us fix $s>0$, $\delta>0$ and introduce Lipschitz cut-off function
$\xi_{s,\delta}(r):\xi_{s,\delta}(r)=0$ for $r\leq s$, $\xi_{s,\delta}(r)=1$ for
$r>s+\delta$, $\xi_{s,\delta}(r)=\frac{r-s}\delta$ for $r:s<r<s+\delta$. Now we
substitute the test function $\eta(t,x)=u(t,x)\xi_{s,\delta}(d(x))$ into integral
identity \eqref{4.1}. Then using the formula of integration by parts (see for example
\cite{AL1}) we get:
\begin{equation}\label{4.6}
\begin{aligned}
&\frac p{p+1}\int_{\Omega(s)}|u(b,x)|^{p+1}\xi_{s,\delta}(d(x))dx+
\int_a^b\int_{\Omega(s)} \Big(\sum_{i=1}^na_i(... ,\nabla_x
u)u_{x_i}+b(t,x)|u|^{\lambda+1}\Big)\xi_{s,\delta}(d(x))dxdt=
\\&=\frac p{p+1} \int_{\Omega(s)}|u(a,x)|^{p+1}\xi_{s,\delta}(d(x))dx
-\int_a^b\int_{\Omega(s)\setminus\Omega(s+\delta)}\sum_{i=1}^na_i(... ,\nabla_x
u)u\xi_{s,\delta}(d(x))_{x_i}dxdt.
\end{aligned}
\end{equation}
Let us take in \eqref{4.6} $b=\tau<T$, $a=s^q$. Then passing to the limit
$\delta\rightarrow0$ and using conditions \eqref{coercit}, \eqref{growth} we derive by
standard computations the following inequality:
\begin{equation}\label{4.7}
h_\tau(s)+E_\tau(s)\leq c_1\int^\tau_{s^q}\int_{\partial\Omega(s)}|\nabla_x
u(t,x)|^p|u|d\sigma dt+c_2h_{s^q}(s) \quad \text{for a.a. }s\in(0,s_\Omega),
\end{equation}
where $c_1<\infty,\,c_2<\infty$ depend on $d_0,\,d_1,\,p,\,n$ only. Let us estimate the
terms in the right hand side from above. Using H\"{o}lder inequality we get:
\begin{equation*}
\begin{aligned}
&\int_{\partial \Omega(s)}|\nabla_xu|^p|u|d\sigma =\int_{\partial \Omega(s)}|u|
g_1(s)^{\frac{1}{\lambda+1}}a_1(t)^{\frac{1}{\lambda+1}}|\nabla_x u|^p
a_1(t)^{-\frac{1}{\lambda+1}}g_1(s)^{-\frac{1}{\lambda+1}}d\sigma\leq
\\&\leq c_3\biggr(\int_{\partial\Omega(s)}|u|^{\lambda+1}a_1(t)
g_1(s)d\sigma\biggr)^{\frac{1}{\lambda+1}}\biggr(\int_{\partial\Omega(s)}|\nabla_xu|^{p+1}d\sigma\biggr)^{\frac{p}{p+1}}
a_1(t)^{-\frac{1}{\lambda+1}}g_1(s)^{-\frac{1}{\lambda+1}},
\end{aligned}
\end{equation*}
where $c_3=(\text{meas }\partial\Omega)^{\frac{\lambda-1}{(\lambda+1)(p+1)}}$.
Integrating the last inequality with respect to $t$ and using H\"{o}lder and Young
inequalities, we derive
\begin{equation}\label{4.8}
\begin{aligned}
&\int^\tau_{s^\beta}\int_{\partial\Omega(s)}|\nabla_xu|^p|u|d\sigma dt\leq
c_4g_1(s)^{-\frac{1}{\lambda+1}}\biggr(\int^\tau_{s^\beta}
a_1(t)^{-\frac{p+1}{\lambda-p}}dt\biggr)^{\frac{\lambda-p}{(\lambda+1)(p+1)}}\times
\\&\times\biggr(\int^\tau_{s^\beta}\int_{\partial\Omega(s)}\big(|\nabla_xu|^{p+1}+
a_1(t)g_1(s)|u|^{\lambda+1}\big)d\sigma
dt\biggr)^{\frac{1+p(\lambda+2)}{(\lambda+1)(p+1)}}.
\end{aligned}
\end{equation}
We estimate the second term of the right hand side of \eqref{4.7} using the monotonicity
of the function $g_1(\cdot)$ and H\"{o}lder inequality:
\begin{equation}\label{4.9}
\begin{aligned}
&h_{s^q}(s)=\int_{\Omega(s)}|u(s^q,x)|^{p+1}a_1(s^q)^{\frac{p+1}{\lambda+1}}
g_1(d(x))^{\frac{p+1}{\lambda+1}}a_1(s^q)^{-\frac{p+1}{\lambda+1}}
g_1(d(x))^{-\frac{p+1}{\lambda+1}}dx\leq
\\&\leq c_5\biggr(\int_{\Omega(s)}|u(s^q,x)|^{\lambda+1}a_1(s^q)g_1(d(x))dx\biggr)^{\frac{p+1}{\lambda+1}}
a_1(s^q)^{-\frac{p+1}{\lambda+1}}g_1(s)^{-\frac{p+1}{\lambda+1}},\quad c_5=(\text{meas
}\Omega)^{\frac{\lambda-p}{\lambda+1}}.
\end{aligned}
\end{equation}
It is easy to check that the following inequality holds:
\begin{equation}\label{4.10}
\begin{aligned}
&-\frac{d}{ds}E_\tau(s)\geq\int^\tau_{s^q}\int_{\partial\Omega(s)}\big(|\nabla_xu|^{p+1}+
a_1(t)g_1(s)|u|^{\lambda+1}\big)d\sigma dt+
\\&+qs^{q-1}\int_{\Omega(s)}\big(|\nabla_xu(s^q,x)|^{p+1}+
a_1(s^q)g_1(d(x))|u(s^q,x)|^{\lambda+1}\big)dx
\end{aligned}
\end{equation}
for almost all $s:0<s<s_\Omega$. Using estimates \eqref{4.8}, \eqref{4.9} and relation
\eqref{4.10} we deduce from \eqref{4.7} inequality \eqref{4.55} with
$C_2=c_2c_5\big(\min_{0\leq s\leq s_0}a_1(s^q)\big)^{-\frac{p+1}{\lambda+1}}$,
$C_1=c_1c_4$.
\end{proof}

Now using the monotonic decreasing of the function $h_\tau(s)$ we derive by the simple
computation from \eqref{4.55} the following inequality:
\begin{equation}\label{4.11}
\begin{aligned}
&-\frac{d}{ds}B_\tau(s)\geq
\min\biggr\{H_\tau^{(1)}(s,B_\tau(s)),H^{(2)}(s,B_\tau(s))\biggr\}\ \ \text{ for a.a. }
s\in (0,s_\Omega),\ \forall\tau<T,
\\&H_\tau^{(1)}(s,B_\tau(s)):=\biggr(\frac{g_1(s)^{\frac{1}{\lambda+1}}B_\tau(s)}
{2C_1\Phi(\tau)^{\frac{\lambda-p}{(\lambda+1)(p+1)}}}\biggr)^{\frac{(\lambda+1)(p+1)}{1+p(\lambda+2)}},\
H^{(2)}(s,B_\tau(s)):=\biggr(\frac{g_1(s)^{\frac{p+1}{\lambda+1}}B_\tau(s)}
{2C_2s^{-\frac{(q-1)(p+1)}{\lambda+1}}}\biggr)^{\frac{\lambda+1}{p+1}}.
\end{aligned}
\end{equation}
Using condition \eqref{4.3} for $q$ we deduce by the standard analysis (see, lemma 2.2
in \cite{Sh1}) that an arbitrary solution of ODI \eqref{4.11} is bounded from above by a
solution of ODE $\frac{d}{ds}B(s)=-H_\tau^{(1)}(s,B(s))$, $B(0)=\infty$, and
consequently for $B_\tau(s)$ from \eqref{4.11} the following estimate holds:
\begin{equation}\label{4.12}
B_\tau(s)\leq C_3\Phi(\tau)G_1(s)\ \ \ \forall s\in(0,s_\Omega),\ \ \forall \tau<T,
\end{equation}
where
$G_1(s)=\biggr(\int^s_0g_1(h)^{\frac{p+1}{1+p(\lambda+2)}}dh\biggr)^{-\frac{1+p(\lambda+2)}{\lambda-p}}$,
$C_3=(2C_1)^{\frac{(\lambda+1)(p+1)}{\lambda-p}}\biggr({\frac{\lambda-p}{1+p(\lambda+2)}}\biggr)^{\frac{1+p(\lambda+2)}{\lambda-p}}$

{\bf Proof of Theorem \ref{Th.2}.} Due to condition \eqref{abs5} for the function
$\Phi(\cdot)$ from \eqref{4.55} the following estimate holds:
\begin{equation}\label{4.13}
\Phi(t)\leq\Phi_0(t):=c_0T\exp\biggr(\frac{\omega(p+1)}{\lambda+1}(T-t)^{-\frac1{p+\mu}}\biggr)\
\ \ \forall t<T.
\end{equation}
Now inequality \eqref{4.12} yields the following estimate:
\begin{equation}\label{4.14}
\begin{aligned}
&h(t,s)+E(t,s):=\int_{\Omega(s)}|u(t,x)|^{p+1}dx+\int_{\frac T2}^t\int_{\Omega(s)}
|\nabla_xu(\tau,x)|^{p+1}dxd\tau\leq C_3G_1(s)\Phi_0(t)
\\&\forall t\in\biggr(\frac
T2,T\biggr),\ \ \forall s:0<s<s'_\Omega:=\min\biggr(s_\Omega,\biggr(\frac
T2\biggr)^\frac1q\biggr),
\end{aligned}
\end{equation}
where $q$ is from \eqref{4.3}. Now we fix some value $\bar{s}\in(0,s'_\Omega)$ and
deduce from \eqref{4.14} the following "initial" energy estimate:
\begin{equation}\label{4.15}
h(t,\bar{s})+E(t,\bar{s})\leq C_3G_1(\bar{s})\Phi_0(t)\quad\forall\,t\in(t_0,T),\
t_0=\frac T2.
\end{equation}
Now we will consider $u(t,x)$ as a solution of equation \eqref{da1} in the domain
$(t_0,T)\times\Omega(\bar{s})$, $t_0=2^{-1}T$. Then using the condition $b(t,x)\geq0$ we
deduce by standard computation the following analog of relation \eqref{energy2} from
lemma \ref{Lem1}:
\begin{equation}\label{4.16}
\begin{aligned}
&h(b,s)+\frac{d_0(p+1)}pE_a^{(b)}(s)\leq h(a,s)+
\biggr(\int_a^b\int_{\partial\Omega(s)}|\nabla_xu(t,x)|^{p+1}d\sigma dt\biggr)^\frac
p{p+1}\times
\\&\times\biggr[k_1\int_a^bh(t,s)dt+
k_2\big(E_a^{(b)}(s)\big)^\theta\Big(\int_a^bh(t,s)dt\Big)^{1-\theta}\biggr]^\frac1{p+1}
\quad\forall s\in(\bar{s},s'_\Omega),
\\&\forall\,a,b:T_0=2^{-1}T\leq a<b<T,\,\theta=(p+1)^{-1},\,s'_\Omega\text{ is from
\eqref{4.14}}.
\end{aligned}
\end{equation}
Then starting from \eqref{4.16} and using condition \eqref{4.15} as the initial
condition for the corresponding systems of ODI we repeat all stages of the proof of
Theorem \ref{Th.1}. As result, using additionally estimate \eqref{4.13} we obtain the
estimate similar to \eqref{main1}:
\begin{equation}\label{4.17}
\begin{aligned}
&h(t,s)+E(t,s)\leq
c_0c_1TC_3G_1(\bar{s})\exp\left(\bar{c}_2\omega^\frac{p+\mu}\mu(s-\bar{s})^{-\frac{p+1}\mu}\right)
\\&\forall\,t\in(2^{-1}T,T),\,\forall\,\bar{s}\in(0,s'_0)
\,\forall\,s\in(\bar{s},s'_0), s'_0\text{ is from \eqref{main1}},
\end{aligned}
\end{equation}
where $\bar{c}_2=c_2\left(\frac{p+1}{\lambda+1}\right)^\frac{p+\mu}\mu$; $c_2$ and $c_1$
are from \eqref{main1}. Optimizing the last estimate with respect to a free parameter
$\bar{s}:0<\bar{s}<s<s'_0$ we get:
\begin{equation}\label{4.18}
h(t,s)+E(t,s)\leq
C_4\min_{0<\bar{s}<s}G_1(\bar{s})\exp\left(\bar{c}_2\omega^\frac{p+\mu}\mu(s-\bar{s})^{-\frac{p+1}\mu}\right),\quad
C_4=c_0c_1TC_3,
\end{equation}
which is desired estimate \eqref{bound08} with $K_1=C_4$, $K_2=\bar{c}_2$.

{\bf Example 1.} Let $g_1(s)=\exp\left(-as^{-\nu}\right)$, $a=const>0$, $\nu=const>0$.
Integrating by parts we easily get the equality:
\begin{equation}\label{4.20}
\int_0^s\exp\left(-bh^{-\nu}\right)\left(1+\frac{\nu+1}{b\nu}\,h^\nu\right)dh=
\frac{s^{\nu+1}}{b\nu}\exp\left(-bs^{-\nu}\right)\quad \forall\,s>0,\,b=const>0.
\end{equation}
Therefore,
\begin{equation}\label{4.21}
\int_0^s\exp\left(-bh^{-\nu}\right)dh\geq b_1s^{\nu+1}\exp\left(-bs^{-\nu}\right)\quad
\forall\,s:\,0<s<\tilde{s}:=\left(\frac{b\nu}{\nu+1}\right)^\frac1\nu,
\end{equation}
where $b_1=(2b\nu)^{-1}$. Now due to \eqref{4.21} with $b=\frac{a(p+1)}{1+p(\lambda+2)}$
we get:
\begin{equation}\label{4.22}
\begin{aligned}
&G_1(s)\leq G_1^{(0)}(s):= B_1s^{-\frac{(\nu+1)(1+p(\lambda+2))}{\lambda-p}}
\exp\left(\frac{a(p+1)}{\lambda-p}\,s^{-\nu}\right)
\\&\forall\,s:\,0<s<\tilde{s}:=\left(\frac{a\nu(p+1)}{(1+p(\lambda+2))(\nu+1)}\right)^\frac1\nu,
\quad
B_1=\left(\frac{2a\nu(p+1)}{1+p(\lambda+2)}\right)^\frac{1+p(\lambda+2)}{\lambda-p}.
\end{aligned}
\end{equation}
and estimate \eqref{4.18} holds with $G_1^{(0)}(\cdot)$ instead of $G_1(\cdot)$.

{\bf Example 2.} Let $g_1(s)=as^{\nu}$, $a=const>0$, $\nu=const\geq0$.

Then
$G_1(s)=a^{-\frac{p+1}{\lambda-p}}\left(1+\frac{\nu(p+1)}{1+p(\lambda+2)}\right)^\frac{1+p(\lambda+2)}{\lambda-p}s^{-\frac{(\nu+1)(p+1)+p(\lambda+1)}{\lambda-p}}$.

\vskip 20 pt

{\bf Acknowledgements.} The research of A.~Shishkov for this publication was supported
by Ministry of Education and Science of Russian Federation (the Agreement
N.02.a03.21.004).

The research of Ye.~Yevgenieva was supported by the Project 0117U006353 from the
Department of Targeted Training of Taras Shevchenko National University of Kyiv at the
NAS of Ukraine.

\bigskip

CONTACT INFORMATION

\medskip
{\bf A.E.~Shishkov,}\\Institute of Applied Mathematics and Mechanics of NASU, Slavyansk,
Ukraine,
\\Peoples' Friendship University of Russia, Moscow, Russia,
\\{\bf aeshkv@yahoo.com}

\medskip
{\bf Ye.A.~Yevgenieva,}
\\Institute of Applied Mathematics and Mechanics of NASU, Slavyansk,
Ukraine,
\\{\bf yevgeniia.yevgenieva@gmail.com}


\vskip 20 pt
\begin{thebibliography}{9}

\bibitem{AL1} H.\,W.~Alt, S.~Luckhaus, Quasilinear elliptic-parabolic differential equations,
Math. Z. \textbf{183}, No~3, 311--341 (1983).

\bibitem{DV1} J.\,I. Diaz, L. Veron, Local vanishing properties of solutions of elliptic
and parabolic quasilinear equations, Trans. Amer. Math. Soc. \textbf{290}, No~2,
787--814 (1985).

\bibitem{KSSh} A.\,A.~Kovalevsky, I.\,I.~Skrypnik and A.\,E.~Shishkov,
Singular Solutions in Nonlinear Elliptic and Parabolic Equations (De Gruyter Series in
Nonlinear Analysis and Applications 24, De Gruyter, Basel, 2016), p.~435.

\bibitem{Stamp} G. Stampacchia,
\'{E}quations elliptiques du second ordre \`{a} coefficients discontinus, {S\'{e}minaire
de Math\'{e}matiques Sup\'{e}rieures, No.~16 (\'{E}t\'{e},~1965)}.
 --- Montreal: Les Press. Univ. Montreal, 1966.

\bibitem{SGKM1} A.\,A.~Samarskii, V.\,A.~Galaktionov, S.\,P.~Kurdyumov, A.\,P.~Mikhailov,
Regimes with peaking in problems for quasilinear parabolic equations (Nauka, Moscow,
1987), p.~480. (in Russian)

\bibitem{GH1} B.\,H.~Gilding, M.\,A.~Herrero,
Localization and blow-up of termal waves in nonlinear heat conduction with peaking,
Math. Ann. \textbf{282}, No~2, 223--242 (1988).

\bibitem{CE1} C.~Cortazar, M.~Elgueta,
Localization and boundedness of the solutions of the Neumann problem for a filtration
equation, {Nonlinear Anal.} \textbf{13}, No~1, 33--41 (1989).

\bibitem{GG1} B.\,H.~Gilding, I.~Goncerzewicz,
Localization of solutions of exterior domain problems for the porous media equation with
radial symmetry, {SIAM J. Math. Ann.} \textbf{31}, No~4, 862--893 (2000).

\bibitem{Ven1} T.\,O.~Venegas,
The porous media equation with blowing up boundary data, {Adv. Nonlinear Stud.},
\textbf{9}, No~1, 1--27 (2009).

\bibitem{GSh1} V.\,A.~Galaktionov, A.\,E.~Shishkov,
Saint-Venant's principle in blow-up for higher order quasilinear parabolic equations,
{Proc. Roy. Soc. Edinburgh. Sect. A} {\bf 133}, No~5, 1075--1119 (2003).

\bibitem{GSh2} V.\,A.~Galaktionov, A.\,E.~Shishkov,
Structure of boundary blow-up for higher-order quasilinear parabolic equations, {Proc.
R. Soc. Lond., Ser. A, Math. Phys. Eng. Sci.} {\bf 460}, No~2051, 3299--3325 (2004).

\bibitem{ShSh1} A.\,E.~Shishkov, A.\,G.~Shchelkov,
Boundary regimes with peaking for general quasilinear parabolic equations in
multidimensional domains, {Math. Sb.} \textbf{190}, No~3-4, 447--479 (1999).
(in~Russian)

\bibitem{DPP1} Y.~Du, R.~Peng, P.~Pola\^{c}ik,
The parabolic logistic equation with blow-up initial and boundary values, {Journal
D'Analyse Mathematique} \textbf{118}, 297--316 (2012).

\bibitem{Sh1} A.~Shishkov,
Large solutions of parabolic logistic equation with spatial and temporal degeneracies,
{DCDS, ser.S} \textbf{10}, No~10, 895--907 (2017).

\end{thebibliography}
\end{document}